\documentclass{amsart}
\usepackage{amsmath,amssymb,amsfonts}
\usepackage{mathtools}
\usepackage{esvect}
\usepackage[shortlabels]{enumitem}
\setlist{leftmargin=*}

\usepackage{amsfonts}
\usepackage[textsize=tiny, color=green]{todonotes} 

\usepackage[colorlinks=true, linkcolor=blue]{hyperref}

\usepackage[sorted]{amsrefs}

\newtheorem{thm}{Theorem}[section]
\newtheorem{cor}[thm]{Corollary}
\newtheorem{lem}[thm]{Lemma}
\newtheorem{prop}[thm]{Proposition}
\newtheorem{claim}[thm]{Claim}
\newtheorem{fact}[thm]{Fact}
\theoremstyle{definition}
\newtheorem{defn}[thm]{Definition}
\theoremstyle{remark}

\newtheorem{rem}[thm]{Remark}

\newtheorem{sample}[thm]{Example} \numberwithin{equation}{section}
    {\medskip\begingroup\leftskip 0.5cm\rightskip 0.5cm\noindent\begin{small}{\bf Remark.}}
    {\end{small}\par\endgroup}
{\begin{list}{$\bullet$}
 {\settowidth{\labelwidth}{\textsf{$\bullet$}} \setlength{\leftmargin}{10pt}}}
{\end{list}}

\newcounter{ssample}[section]

{\color{Bittersweet} \noindent Example \refstepcounter{ssample}\hbox{\bf \arabic{section}.\arabic{ssample}.}}
{}

\newcounter{insertcount}

    {\color{blue} \medskip\begingroup\noindent\begin{small}{\color{blue} \stepcounter{insertcount}
          {
            \bf Insert \arabic{insertcount}. #1.}
            \addcontentsline{toc}{subsection}{{\ \ \small  Insert \arabic{insertcount}: #1}}
               \leavevmode  }
           }
    {\end{small}\par\endgroup}

\newcommand{\mrmk}[1]
{{\tiny$^{\spadesuit}$}\marginpar{\fbox{\footnotesize #1}}}
\def\strutdepth{\dp\strutbox}%
\def\marginalnote#1{\strut\vadjust{\kern-\strutdepth\specialnote{#1}}}%
\def\specialnote#1{\vtop to \strutdepth{\baselineskip%
\strutdepth\vss\llap{\hbox{\scriptsize \bf #1}}\null}}%

\newcommand{\RR}{\mathbb{R}}

\newcommand{\QQ}{\mathbb{Q}}

\def\CF{\mathcal F}

\newcommand{\CR}{\mathcal R} 
\newcommand{\NN}{\mathbb N}

\newcommand{\co}{\circ}

\newcommand{\1}{\mathbf{1}}

\reversemarginpar

\def\CM{\mathcal{M}}

\newcommand*\bbar[1]{%
  \hbox{%
    \vbox{%
      \hrule height 0.5pt 
      \kern0.5ex
      \hbox{%
        \kern-0.1em
        \ensuremath{#1}%
        \kern-0.1em
      }%
    }%
  }%
} 

\gdef\mh{\text{-}}

\gdef\Av{\operatorname{Av}}
\gdef\CP{\mathcal{P}}

\gdef\CB{\mathscr{B}}

\gdef\CB{\mathcal{B}}

\gdef\CP{\mathcal{P}}
\gdef\CM{\mathcal{M}}
\gdef\dcup{\,\dot{\cup}\,}
\gdef\ttimes{{\times}}
\gdef\tttimes{{\ttimes}\dotsb\ttimes}
\gdef\ootimes{{\otimes}}
\gdef\oootimes{{\ootimes}\dotsb\ootimes}
\gdef\lltimes{{\ltimes}}
\gdef\llltimes{{\lltimes}\dotsb\lltimes}

\gdef\co{\textrm{c}}

\title{Definable regularity lemmas for NIP hypergraphs}

\author{Artem Chernikov} \address{Department of Mathematics, 
University of California Los Angeles,
Los Angeles, CA 90095-1555} \email{chernikov@math.ucla.edu}

\author{Sergei Starchenko} \address{Department of Mathematics, University of Notre Dame, Notre Dame,
  IN 46556} \email{Starchenko.1@nd.edu}

\RequirePackage[normalem]{ulem} 
\RequirePackage{color}\definecolor{RED}{rgb}{1,0,0}\definecolor{BLUE}{rgb}{0,0,1} 
\providecommand{\DIFaddbegin}{} 
\providecommand{\DIFaddend}{} 

\begin{document}

{\abstract We present a systematic study of the regularity phenomena for NIP hypergraphs and connections to the theory of (locally) generically stable measures, providing a model-theoretic hypergraph version of the results from \cite{alon2007efficient, lovasz2010regularity}. We  also consider the two extremal cases of regularity for stable and distal hypergraphs, improving and generalizing the results from \cite{distal} and \cite{ms}. Finally, we consider a related question of the existence of large (approximately)
  homogeneous definable subsets of NIP hypergraphs and provide some positive results and counterexamples, in particular for graphs definable in the $p$-adics.}

\maketitle

\section{Introduction} 
Szemer\'edi's regularity lemma is a fundamental result in (hyper-)graph combinatorics with numerous applications in extremal combinatorics, number theory and computer science (see
\cite{komlos1996szemeredi} for a survey). We recall it in a simplified  form.
 By a graph $G=(V,E)$ we mean a set $V$ with a symmetric subset
$E\subseteq V^2$. For $A,B\subseteq V$ we denote by $E(A,B)$ the set
of edges between $A$ and $B$, i.e. $E(A,B)=E\cap (A\ttimes B)$.  Given $M \in \mathbb{N}$, we write $[M]$ to denote the set $\{1, 2, \ldots, M \}$.

\begin{fact}[Szemer\'{e}di regularity lemma]\label{lem:sz-reg-graph}

For every real $\varepsilon > 0$ there exists some constant $M = M(\varepsilon) \in \mathbb{N}$ satisfying the following.
Let $G=(V,E)$ be an arbitrary finite graph. Then there is a
partition $V=V_1\cup\dotsb\cup V_M$ into disjoint sets, real numbers
$\delta_{ij}, i,j \in [M]$, and an exceptional set of pairs $\Sigma \subseteq
[M]\ttimes [M]$ such that 
\[ \sum_{(i,j)\in \Sigma}  |V_i||V_j| \leq \varepsilon|V|^2  \]
and for each $(i,j)\in [M]\ttimes [M] \setminus \Sigma$ we have 
\[  | \, |E(A,B)|- \delta_{ij}|A||B|\, | <\varepsilon
|V_i||V_j| \]
for all $A\subseteq V_i$, $B\subseteq{ V_j}$.
\end{fact}

The bounds on $M(\varepsilon)$ are known to be extremely bad: Gowers had demonstrated that it  grows as an exponential tower of height polynomial in $(\frac{1}{\varepsilon})$ (see e.g. \cite{moshkovitz2013short}).

Several recent results demonstrate that better bounds and stronger regularity can be obtained for certain families of hypergraphs satisfying additional combinatorial restrictions. For example, in \cite{fox2012overlap, fox2015polynomial} it is shown that when the edge relation is semialgebraic, of bounded description complexity, then the size of the partition can be bounded by a polynomial in terms of $\frac{1}{\varepsilon}$, all good pairs are actually homogeneous, and the sets in the partition can be chosen to be semialgebraic, of bounded complexity. Similar polynomial bounds were obtained by Tao \cite{tao2012expanding} for algebraic hypergraphs of bounded description complexity in large finite fields and by Alon, Fischer, Newman \cite{alon2007efficient} and Lov{\'a}sz, Szegedy \cite{lovasz2010regularity} for graphs of bounded VC-dimension. 

These results can be naturally viewed as results about hypergraphs with the edge relation definable, in the sense of first-order logic, in certain tame structures, and the restrictions on the complexity of the edge relation in all of the results above are surprisingly well aligned with generalized stability and classification in model theory.
For example, as demonstrated in \cite{distal} (see also \cite{simon2016note}), the results in \cite{fox2012overlap, fox2015polynomial} can be generalized to graphs definable in arbitrary \emph{distal} structures (see Section \ref{sec: distal case}), and that moreover this strong form of regularity characterizes distality. Here ``semialgebraic graphs'' corresponds to the special case of ``graphs definable in the field of reals'', but the result also applies to graphs definable in the $p$-adics, for example. Similarly, the result in \cite{tao2012expanding} can be viewed as a result about graphs definable in \emph{pseudofinite fields}, and admits a natural model theoretic proof and generalizations \cite{pillay2013remarks, HrushovskiUnpubl, garcia2015pseudofinite}. Another very important example is given by the regularity lemma for \emph{stable graphs} \cite{ms} (model-theoretic stability is the notion of tameness at the core of Shelah's classification \cite{ShelahCT}, see Section \ref{sec: stable case}). Similarly, the results in \cite{lovasz2010regularity} can be interpreted as results about graphs definable in \emph{NIP structures} (see below).

Another point of view on the hypergraph regularity phenomenon is through the prism of probability theory. Namely, the existence of a regular partition can be viewed as a finitary version of the existence of the conditional expectation. There are several proofs of the hypergraph regularity lemma in the literature making this precise by reducing working with a family of finite graphs to working with some kind of an analytic ``limit object'' equipped with a probability measure (see \cite{lovasz2007szemeredi, elek2007limits, CDM26}). 

Similarly, regularity for restricted families of graphs can be viewed as the study of (finitely additive) probability measures on certain restricted families of Boolean algebras. Such measures in the model-theoretic setting of Boolean algebras of definable sets were introduced by Keisler \cite{keisler1987measures}, and recently the study of Keisler measures has attracted a lot of attention, especially the study of  \emph{generically stable measures} in \emph{NIP structures} \cite{hrushovski2011nip, hrushovski2013generically, Bourbaki, CheSurv}. The class of NIP structures  was introduced by Shelah in his work on the classification
program \cite{ShelahCT}}. It contains all stable and $o$-minimal structures, along with other important algebraic examples, and  we refer to
\cite{adler2008introduction, simon2015guide} for an introduction to the area (see also Section \ref{sec: hypergraphs definable in NIP} for the definition and some examples). The study of Keisler measures in NIP structures can be viewed as a model theoretic counterpart of the Vapnik-Chervonenkis theory \cite{vc}, and generically stable measure are those Keisler measures that satisfy a form of the VC-theorem for all uniformly definable families (see Section \ref{sec: hypergraphs definable in NIP}).

The connection between the study of generically stable measures in model theory and regularity lemmas for definable hypergraphs was pointed out in the distal case in \cite{distal}, and the aim of this article is to systematically develop these connections for the general (local) NIP setting.

In Section \ref{sec:setting} we give a decomposition result for products of finitely additive probability measures that are well-approximated by counting measures (which we call \emph{finitely approximated measures}, see Section \ref{sec:qiant-free-keisl}), with and without the assumption of finite VC-dimension. Namely, assume we are given some sets $V_1, \ldots, V_k$ equipped with Boolean algebras $\CB_1, \ldots, \CB_k$ of subsets and finitely additive probability measures $\mu_1, \ldots, \mu_k$ on them. Let $R \subseteq V_1 \times \ldots \times V_k$ be an edge relation such that all of its fibers are measurable. It then follows from the finite approximation assumption that there is a Boolean algebra $\CB$ of subsets of $V_1 \times \ldots \times V_k$ extending the product Boolean algebra $\CB_1 \otimes \ldots \otimes \CB_k$ with $R \in \CB$, and such that $\CB$ can be equipped with a natural product measure $\mu$ satisfying a Fubini property (Section \ref{sec:qiant-free-keisl}). Moreover, relatively to $\mu$, the set $R$ can be approximated by a union of boxes (i.e. sets of the form $A_1 \times \ldots \times A_k$ with $A_i \in \CB_i$) up to measure $\varepsilon$, for any real $\varepsilon > 0$, and in the finite VC-dimension case the number of boxes needed is polynomial in $\frac{1}{\varepsilon}$ (Theorem \ref{thm:main}). On the one hand, this can be viewed as a version of the results  for graphons from \cite{lovasz2010regularity}  in a setting better suited for the model-theoretic applications, and generalized to hypergraphs. On the other hand,  
this result can also be viewed as developing elements of the \emph{local} theory of generically stable measures, and refining some of the results in \cite{hrushovski2013generically} for such measures. In our setting, instead of working with Borel measures on the space of types, we use directly the (equivalent) theory of integration for finitely additive measures (sometimes called the theory of charges \cite{charges}), and we give some details for the sake of exposition. Note that we are only assuming bounded VC-dimension on $R$-definable sets, and our definition of a finitely approximated 
  measure is weaker than the definition of fim measures in \cite{hrushovski2013generically} (see Remark \ref{rem: fap vs fim}), so we have to redefine the product of finitely approximated measures.

In Section \ref{sec:appl-hypergr} we apply these results to obtain a \emph{definable} regularity lemma for hypergraphs of bounded VC-dimension, in particular for hypergraphs definable in an NIP structure, uniformly over all generically stable measures. 
For a general $k$-ary hypergraph $(V,E)$ with $E\subseteq{V\choose k}$ for some $k\geq 2$, with $V$ a large finite set, the hypergraph regularity lemma 
\cite{nagle:MR2198495,rodl:MR2069663,MR2373376} allows to represent the characteristic function $\chi_E: {V\choose k} \to \{0,1\}$ of the hyperedge relation in the form 
\[\chi_E=f_{k-1}+\cdots+f_1+f^\bot\]
where: $f_1$ has the form $\sum_{i_1,\ldots,i_k}\alpha_{i_1,\ldots,i_k}\prod_j\chi_{V_{i_j}}(x_j)$ for some partition $V=\bigcup_{i\leq n}V_i$ and some real numbers $\alpha_{i_1, \ldots, i_k}$ (i.e.~a partition of the vertices with weights $\alpha_{i_1, \ldots, i_k}$ indicating the density of edges on $V_{i_1} \times \ldots \times V_{i_k}$); the $f_j$ in general are sums of $j$-ary cylinder sets (for instance, $f_2$ is, roughly speaking, the portion of $\chi_E$ which can be described using directed graphs, $f_3$ using directed $3$-hypergraphs, etc.~up to $k-1$); and $f^\bot$ is the quasi-random $k$-ary function (representing the random determination of which which hyperedges of $E$ are actually present). In Theorem \ref{thm:main1} we show that if $E$ has small VC-dimension then all summands of this decomposition except for $f_1$ are small, hence $E$ is approximated by a union of boxes.
More precisely, for each $k,d$ and $\varepsilon>0$, there is a bound $N = O_{k,d} \left( \left(\tfrac{1}{\varepsilon}\right)^{4(k-1)d^2} \right)$ so that whenever $(E,V)$ is a $k$-ary hypergraph with VC-dimension at most $d$, there is a partition of $V$ into $N$ parts so that $E$ is given, up to symmetric difference of measure $< \varepsilon$, by a union of boxes of the form $V_{i_1} \times \ldots \times V_{i_k}$. Moreover, each of the sets $V_i$ in the partition is $E$-definable, i.e.~given by a Boolean combination of the fibers of $E$ of size bounded in terms of $d,k,\varepsilon$. 

In Section \ref{sec: stable and distal} we discuss regularity in two extreme opposite special cases of the NIP hypergraphs. Namely, we generalize and improve the aforementioned stable \cite{ms, malliaris2016stable} and distal \cite{distal} regularity lemmas in our setting. The (global) model-theoretic implications of these results can be summarized as follows.

\begin{thm} \label{thm: everything}
\begin{enumerate}

\item (Corollary \ref{cor: reg in NIP}) Let $\CM$ be an NIP structure and $k \geq 2$. For every definable relation $E\left(x_{1},\ldots,x_{k}\right)$
there is some $c=c\left(E\right)$ such that for any $\varepsilon>0$
and any generically stable\textcolor{red}{{} }Keisler measures $\mu_{i}$
on $M^{|x_{i}|}$ there are partitions $M^{|x_i|}=\bigcup_{j<K}A_{i,j}$
and a set $\Sigma\subseteq\left\{ 1,\ldots,K\right\} ^{k}$ such that:
\begin{enumerate}
\item $K\leq\left(\frac{1}{\varepsilon}\right)^{c}$.
\item $\mu\left(\bigcup_{\left(i_{1},\ldots,i_{k}\right)\in\Sigma}A_{1,i_{1}}\times\ldots\times A_{k,i_{k}}\right)\leq \varepsilon$,
where $\mu=\mu_{1}\otimes\ldots\otimes\mu_{k}$,
\item for all $\vec{i}=\left(i_{1},\ldots,i_{k}\right) \notin \Sigma$ we have $$\left|\mu \left(E \cap (A_{1,i_1} \times \ldots \times A_{k,i_k}) \right) - \delta_{\vec{i}}\mu(A_{1,i_1} \times \ldots \times A_{k,i_k}) \right| < \varepsilon \mu(A_{1, i_1} \times \ldots \times A_{k, i_k})$$ 
for some $\delta_{\vec{i}} \in \{ 0,1\}$.
\item each $A_{i,j}$ is defined by an instance of an $E$-formula depending
only on $E$ and $\varepsilon$.

\end{enumerate}

	\item (Corollary \ref{cor: definable stable regularity}) Assume that $\CM$ is stable. Then, in addition: 
	\begin{enumerate}
	\item we can take the $\mu_i$'s to be arbitrary Keisler measures (as all measures are automatically generically stable in this case),
	\item we may assume that $\Sigma = \emptyset$, i.e.~all tuples in the partition are $\varepsilon$-regular.
	\end{enumerate}
\item (Theorem \ref{thm: distal regularity}) Assume that $\CM$ is distal. Then in addition we have:
\begin{enumerate}
	\item for all $\left(i_{1},\ldots,i_{k}\right)\notin \Sigma$, either $\left(A_{1,i_{1}}\times\ldots\times A_{k,i_{k}}\right)\cap E=\emptyset$
or $A_{1,i_{1}}\times\ldots\times A_{k,i_{k}}\subseteq E$,

	\item if the relation $E$ is defined by an instance of a formula $\theta$, then we can take each $A_{i,j}$ to be defined by an instance of a formula $\psi_{i}\left(x_{i},z_i \right)$
which only depends on $\theta$ (and not on $\varepsilon$).

	\end{enumerate}
\end{enumerate}

\end{thm}

Finally, in Section \ref{sec: large homog set}, we consider a related question of the existence of large (approximately) homogeneous definable subsets of definable NIP hypergraphs (i.e., the measure theoretic versions of the results of Erd\H os, Hajnal and R\"odl, see e.g. \cite{fox2008induced}). As a corollary of the regularity lemma, we show  that for every $d$ and $\alpha, \varepsilon>0$ there is some $\delta = \delta(d, \alpha, \varepsilon) > 0$ such that the following holds. Let  a hypergraph $R \subseteq V_1 \times \ldots \times V_k$ of VC-dimension at most $d$ be given, and  let  $\mu_i$ be measures on $V_i$ which are all finitely approximated on $R$. Assume that the density of $R$ on $V_1 \times \ldots \times V_k$ (relatively to the product measure) is at least $\alpha$. Then it is possible to find $R$-definable sets $A_i \subseteq V_i$ such that $\mu_i(A_i) \geq \delta$ and such that the density of $R$ on $A_1 \times \ldots \times A_k$ is at least $1 - \varepsilon$  (Theorem \ref{thm: density approx EH}).
 The situation is quite different in the non-partitioned case. Namely, when $V = V_1 = \ldots = V_k$, $\mu = \mu_1 = \ldots = \mu_k$ and $R$ is a symmetric relation, we would like to find a definable subset $A$ of $V$ of positive measure, such that the density of $R$ on $A$ is $\varepsilon$-close to $0$ or $1$ (the result above applied to this situation would typically produce \emph{disjoint} sets $A_1, \ldots, A_k$). A classical theorem of R\"odl (see Fact \ref{fac: Rodl}) implies that this is indeed possible for pseudofinite counting measures, with all internal sets added to the language. We provide an example of a definable graph in the $p$-adics which does not admit uniformly definable sets of positive measure with this property, relatively to the additive Haar measure (Section \ref{sec: p-adics counterex}) (hence demonstrating that an analogue of R\"odl's theorem does not hold for finitely approximated measures in general).

\subsubsection*{Related work}
The results of this paper originally appeared in a July 2016 preprint \cite{chernikov2016definable}. In a preprint \cite{fox2017erdos} from October 2017 (which has later appeared as \cite{fox2019erdHos}), its authors obtain a stronger bound on the size of the partition for hypergraph regularity with finite VC-dimension in the case of uniform measures on finite spaces (corresponding to Corollary \ref{lem:sze-hyper}) and demonstrate its optimality.
Independently from our work in the present paper, in  a preprint \cite{ackerman2017stable} its authors obtain a version of the stable case of Theorem \ref{thm: everything} for finite hypergraphs (and more generally, finite relational structures).
The NIP hypergraph regularity lemma is further generalized to $n$-dependent hypergraphs in \cite{chernikov2020hypergraph}.

Recall that a partition is \emph{equitable} if any two sets in it have the same measure (possibly up to a rounding error). We remark that it is not always possible to choose an \emph{$E$-definable} equipartition with the additional properties discussed above in the NIP, stable and distal cases. However,
restricting to finite hypergraphs with uniform finitely supported measures and giving up definability, this can be achieved in each of the three cases. See \cite{fox2019erdHos} for the finite VC-dimension/NIP hypergraphs; \cite{ms} for stable graphs and \cite{ackerman2017stable} for stable hypergraphs; and \cite[Section 5.3]{distal} (along with \cite[Section 3]{simon2016note}) for distal hypergraphs (where in fact a definable equitable partition can be chosen in many cases).

\subsubsection*{Notation} Given $r,s, \delta \in \mathbb{R}$ with $\delta \geq 0$, we write $r \approx^{\delta} s$ if $|r - s| \leq \delta$.

\subsection*{Acknowledgements}
 We would like to thank the referee, Matthias Aschenbrenner and Roland Walker for their very helpful comments on the earlier versions of the article.
 Chernikov was supported by the NSF Research Grant DMS-1600796, by the NSF CAREER grant DMS-1651321 and by an Alfred P.~Sloan Fellowship. Starchenko was  supported by the NSF Research Grant DMS-1500671.

\section{Decomposing product measures} 
\label{sec:setting}
In this section, we present some general results on decomposing products of finitely additive probability measures that can be locally approximated by frequency measures.

\subsection{Notation}
We will use the following notation:
\begin{itemize}
\item  For $k\in \NN$ we will denote by $[k]$ the set
  $\{1,\dotsc,k\}$. 
\item For an integer $k$ and $I\subseteq [k]$ we will denote by $I^\co$
  the complement $I^\co=[k]\setminus I$. For $i\in [k]$ instead of
  $\{i\}^\co$ we write $i^\co$. 
\item For sets $V_1,\dotsc, V_k$ and $I\subseteq [k]$ we denote by $V_I$ the product 
$V_I=\prod_{i\in I} V_i$.   
\item  Let  $R \subseteq V_1\tttimes V_k$ be a $k$-ary relation and $I\subseteq [k]$. Viewing
  $R$ as a binary relation on  $V_I\ttimes V_{I^\co}$, 
for  $b\in V_{I^\co}$
we denote by $R_{b}$ the fiber 
\[ R_{b} =\{ a \in V_I \colon (a,b) \in R\}. \]
\end{itemize}

\begin{defn} \DIFaddbegin \label{defn:const}
 \DIFaddend Let $V_1,\dotsc,V_k$ be  sets, $R\subseteq
 V_1\tttimes V_k$ and  $I \subseteq  [k]$.  
 \DIFaddbegin 

 \begin{enumerate}
 \item \DIFaddend We say that a subset
 $X\subseteq V_I$ is \emph{$R$-definable over a set $D\subseteq V_{I^\co}$}
if it is a finite Boolean combination of sets of the form 
$R_{b}$ with $b \in D$, and say that $X$ is \emph{$R$-definable} if it is
$R$-definable over $V_{I^\co}$.    
 \item We say
that a set  $A\subseteq V_1\tttimes V_k$  is \DIFaddbegin \emph{\DIFaddend $R_{\otimes}$-definable} if $A$ can be written as a
finite union of sets of the form $X_1\tttimes X_k$, such that each $X_i \subseteq 
V_i$ is $R$-definable. 
\item In addition, given a tuple $\vec D =(D_1,\dotsc, D_k)$ with $D_i \subseteq V_{i^\co}$,  we say that $A$
   is $R_\ootimes$-definable over $\vec D$ if  every $X_i$ above is $R$-definable over $D_i$.  Note that the subsets of $V_1 \times \ldots \times V_k$ which are $R_\ootimes$-definable over $\vec{D}$ form a Boolean algebra.
 For such a tuple $\vec D$, we define 
$\|\vec D\| :=\max\{|D_i| \colon i\in [k]\}$.

 \end{enumerate}

\DIFaddend \end{defn}

We recall the notion of VC-dimension (see e.g. \cite[Chapter 10]{matouvsek2002lectures}). Let $V$ be a set, finite or infinite, and let $\mathcal{F}$ be a
family of subsets of $V$.  Given $A \subseteq V$, we say that it is \emph{shattered} by
$\mathcal{F}$ if for every $A' \subseteq A$ there is some $S \in \mathcal{F}$ such that
$A \cap S = A'$. The \emph{VC-dimension of $\mathcal{F}$}, that we will denote by $VC(\mathcal{F})$, is the smallest integer
$d$ such that no subset of $V$ of size $d+1$ is shattered by $\CF$.  For a set $B\subseteq V$, let
$\mathcal{F}\cap B=\left\{ A\cap B:A\in\mathcal{F}\right\}$. The \emph{shatter function of $\CF$} is  defined as 
$\pi_{\mathcal{F}}\left(n\right)=\max\left\{ \left|\mathcal{F}\cap B\right|:B\subseteq
  V,\left|B\right|=n\right\} $.

\begin{fact}[Sauer-Shelah lemma]\label{fac: SauerShelah} If $VC(\mathcal{F}) \leq d$ then for
  $n\geq d$ we have
  $\pi_{\mathcal{F}}\left(n\right)\leq\sum_{i\leq d}{n \choose i}=O\left(n^{d}\right)$.
\end{fact}

\begin{defn}\label{def:nip}
For sets  $V_1,\dotsc,V_k$ and 
a set 
$R\subseteq V_1\tttimes V_k$ we say that $R$ has \emph{VC-dimension at most
$d$} if for every $i \in [k]$ the family $\{ R_a \colon a\in
V_{[k] \setminus \{ i\}}\}$ of subsets of $V_i$ is a family with VC-dimension at most $d$. 
\end{defn}

The next fact follows from the Sauer-Shelah Lemma.

\begin{fact}
  \label{fact:s-sh}
For every $d\in \NN$ there is a constant $C_d$ such that for any
relation $R\subseteq V\times W$ of VC-dimension at
most $d$ and any finite $D\subseteq V$, the number of atoms in the Boolean algebra of subsets of $W$ which are $R$-definable over $D$ is at most $C_d |D|^d$.  
\end{fact}

\subsection{Basics  on Boolean algebras and measures.}
\label{sec:compr-syst-bool}

Recall that for a set $V$, \emph{a field on $V$} is a Boolean algebra  
of subsets of $V$. 

For sets $V_1,\dotsc,V_k$ and fields 
$\CB_i$ on $V_i$, $i\in [k]$, as usual, we denote by
$\CB_1\oootimes\CB_k$ the field on 
$V_1\tttimes V_k$ generated by the sets
$X_1\tttimes X_k$ with $X_i\in \CB_i$. It is not hard to see
that every set in $\CB_1\oootimes \CB_k$ is a disjoint
union of sets of the form $X_1\tttimes X_k$ with $X_i\in \CB_i$. Given $I = \{ i_1, \ldots, i_n \} \subseteq [k]$, we let $\CB_{I} := \bigotimes_{i \in I} \CB_i = \CB_{i_1} \otimes \ldots \otimes \CB_{i_n}$.

\subsubsection{Finitely additive probability measures}
\label{sec:prel-finit-addit}

\begin{defn}
Let $V$ be a set and $\CB$ be a field on $V$. In this paper, a \emph{measure
  on $\CB$}  
  is a finitely additive probability measure on $\CB$,
i.e.~a function $\mu\colon \CB\to \RR^{\geq 0}$ such that $\mu(\emptyset)=0$,
$\mu(V)=1$  and
$\mu(A\cup B)=\mu(A)+\mu(B)-\mu(A\cap B)$ for all $A,B\in \CB$.
\end{defn}

Let $V_1,\dotsc,V_k$ be sets and $\CB_i$ be fields on $V_i$, $i\in [k]$. 
Assume we have a measure  $\mu_i$ on $\CB_i$ for each $i\in[k]$. 
It is not hard to see that there  is a unique measure $\mu$ on $\CB_1\oootimes\CB_k$ with 
$\mu(A_1\tttimes A_k)=\prod_{i=1}^k \mu_i(A_i)$ for all  $A_i\in
\CB_i$, $i\in [k]$. 
 We will denote this measure   $\mu$ by $\mu_1{\tttimes}\mu_k$.

\subsubsection{Integration with respect to finitely additive measures}
\label{sec:integr-with-resp-1}

We will need some basic facts about integration relatively to finitely additive measures (we refer to \cite{charges} for a detailed account).

As usual for a set $V$ and a subset $X\subseteq V$ we will denote by
$\1_X$ the indicator function of $X$ on $V$.  

We fix a set $V$ and a field $\CB$  on $V$.

We say that a  function $f\colon V\to
\RR$ is \emph{$\CB$-simple}  if there are $X_1,\dotsc,X_n\in
\CB$ and $r_1,\dotsc,r_n\in \RR$ with $f=\sum_{i=1}^n r_i \1_{X_i}$.   
Obviously the set of all $\CB$-simple functions forms an
$\RR$-algebra.

For a measure  $\mu$ on $\CB$ and  a $\CB$-simple  function $f =\sum_{i=1}^n r_i \1_{X_i}$ 
we define 
\[ \int_V f\,d\mu = \sum_{i=1}^{n} r_i\mu(X_i). \]
It is easy to see that the above integral does not depend on a
representation of $f$ as a simple  function. If a subset
$A\subseteq V$ is in $\CB$ then we also define 
\[ \int_A f\,d\mu =  \int_V (\1_Af)\,d\mu = \sum_{i=1}^{n} r_i\mu(A\cap X_i). \]

\begin{rem}\label{rem:mu-as-int}
Clearly for $A\in \CB$ we have $\mu(A)=\int_V  \1_A \,d\mu $. 
\end{rem}

We say that a function $f\colon V\to \RR$ is \emph{$\CB$-integrable}, or
just \emph{integrable}, if it is 
in the closure of the set of  $\CB$-simple  functions with respect to the
$L_\infty$-norm, i.e. for all $\varepsilon > 0 $  there is a
$\CB$-simple 
function $g$
with $|f(x)-g(x)| < \varepsilon$ for all $x\in V$.
The following claim is obvious.

\begin{claim}  
\label{lem:def-sep}
A function $f\colon V\to \RR$ is $\CB$-integrable if and only if 
 for any $\varepsilon>0$ there are   $Y_1,\dotsc,Y_{n}\in \CB$
covering $V$ such that for any $i\in[n]$ and any $c,c'\in Y_i$ we
    have $|f(c)-f(c')| < \varepsilon$. 
\end{claim}

If $f$ is $\CB$-integrable and $\mu$ is a measure  on $\CB$ then the integral of
$f$ with respect to $\mu$ is defined as 
\[ \int_V f d\mu =\lim_{n\to \infty} \int_V g_n d\mu, \]
where $(g_n)_{n\in \NN}$ is a sequence of $\CB$-simple  functions convergent to $f$. It
is easy to see that this integral does not depend on the choice
of a convergent
sequence.  
Also for a $\CB$-integrable function $f$ and a set $A\in \CB$ we
define
 \[ \int_A f\,d\mu =  \int_V (\1_Af)\,d\mu. \]

\subsection{On $\varepsilon$-nets}
\label{sec:varepsilon-nets}

Let $V$ be a set, $\CB$ a field on $V$ and $\mu$ a measure on
$\CB$. Let $\CF$ be a family of subsets of $V$ with $\CF\subseteq
\CB$. As usual, for $\varepsilon>0$ we say that a subset $T\subseteq
V$ is \emph{an $\varepsilon$-net for $\CF$  with respect to $\mu$ }
if for every $F\in \CF$ we
have $\mu(F)\geq  \varepsilon \Longrightarrow F\cap T \neq \emptyset$.  

The following is a well-known consequence of the classical VC-theorem (see
\cite{vc,vc1} and also \cite{lovasz2010regularity}).

\begin{fact}\label{fact:vc-theorem} Let $V$ be a set, $\CB$ a field on $V$ and $\mu$ a  measure on
$\CB$ with a finite support (i.e. there exists a finite set $A\in \CB$ with $\mu(A)=1$).  If $\CF\subseteq \CB$ is a VC-family with
VC-dimension at most $d$ then for any $\varepsilon>0$ it admits an $\varepsilon$-net
$T$ with $|T|\leq 8d  \frac{1}{\varepsilon}\log{\frac{1}{\varepsilon}}$. 

\end{fact}

\subsection{Finitely approximated measures}
\label{sec:qiant-free-keisl}
Throughout this section we let $V$ be a set and  $\CB_V$ a field on $V$.
 
\DIFaddbegin \begin{defn} Let  $\mu$ be a
  measure on $\CB_V$,   and  let $\mathcal{F} \subseteq \CB_V$ be a family of subsets of $V$.
We say that $\mu$ is \emph{finitely approximated} on the family  $\mathcal{F}$ if for
  every $\varepsilon >0$  there are $p_1,\dotsc,p_n\in V$ (possibly with repetitions) with
 \[ | \mu(F) - \Av(p_1,\dotsc,p_n; F)| < \varepsilon \text{ for every }
F\in \CF, \]
 where $\Av(p_1,\dotsc,p_n; F)=\frac{1}{n}\bigl|\{ i\in
[n] \colon  p_i\in F\}\bigr|$. We say that $p_1, \ldots, p_n$ is an \emph{$\varepsilon$-approximation} of $\mu$ on $\CF$.
\end{defn}

\textbf{Notation.} Let $\mathcal{F}$ be a family of subsets of $V$. For $m \in \mathbb{N}$, we let $\mathcal{F}^m$ be the family of all subsets of $V$ given by the Boolean combinations of at most $m$ sets from $\mathcal{F}$.

\DIFaddend \begin{defn} \label{def: fap on R} Let now $W$ be another set, and let $R\subseteq V\times W$ be a relation such that $R_b \in \CB_V$ for all $b \in W$. 
 \DIFaddbegin \begin{enumerate}
\item Let $\mathcal{R}_V := \{ R_b : b \in W \}$, then $\mathcal{R}_V \subseteq \CB_V$ by assumption.
\item  We say that a measure  $\mu$ on $\CB_V$ is \emph{finitely approximated on the relation $R$} if it is finitely approximated on each of the (infinitely many) families of sets $\mathcal{R}_{V}^{m}$, $m \in \mathbb{N}$.
\end{enumerate}

\DIFaddend \end{defn}

\begin{rem}
\begin{enumerate}
\item In particular, if $\mu$ is finitely approximated on the relation  $R$, then it is finitely approximated on the family of sets $\mathcal{R}_{V}^\Delta = \{ R_b \Delta R_{b'} : b,b' \in W \}$.
\item Note that $\mu$ being finitely approximated on $\mathcal{R}_V^\Delta$  does not imply that $\mu$ is finitely approximated on $\mathcal{R}_V$. For example, let $V = \mathbb{R}$, let $\mathcal{B}_V$ be the Boolean algebra generated by all intervals in $V$, and let $\mathcal{R}_V$ be the family of all intervals unbounded from above. Let $\mu$ be the $0-1$ measure on $\mathcal{B}_V$ such that the measure of a set is $1$ if and only if it is unbounded from above. Then all sets in $\mathcal{R}_V^\Delta$ have measure $0$, so we can take the empty set as an $\varepsilon$-approximation for $\mu$ on $\mathcal{R}_V^\Delta$, for any $\varepsilon > 0$.
But there are no finite $\varepsilon$-approximations for $\mu$ on $\mathcal{R}_V$, for any $\varepsilon<1$, as any finite set can be avoided by some unbounded interval of measure $1$.

Similarly, if $\mathcal{R}_V$ is the family of all intervals bounded from above, then $\mu$ is trivially finitely approximated on $\mathcal{R}_V$. However, it is not finitely approximated on the family of all complements of the sets in $\mathcal{R}_V$.
\end{enumerate}

See Example \ref{ex: generically stable measures} for many examples of finitely approximated measures.

\end{rem}

 \begin{claim}
\label{claim:fim-int}
Let $R \subseteq V \times W$ be as in Definition \ref{def: fap on R},  $\mu$ a measure on $\mathcal{B}_V$, and let $\CB_W$ be a Boolean algebra on  $W$ such that $R_a \in \mathcal{B}_W$ for all $a \in V$.
\DIFaddbegin 

Assume that $\mu$ is finitely approximated on the family  $\mathcal{R}_V^\Delta$.  Then for any set $A \in \CB_V$, the function
$$h_{R,A}: W \to \mathbb{R}, \text{ given by }  h_{R,A}(b)=\mu(R_b \cap A)$$

is 
$\CB_W$-integrable.
\end{claim}
\begin{proof}  Let $\varepsilon>0$.  By assumption we can choose   $p_1,\dotsc,p_n\in V$ such that 
$$ | \mu(R_b \Delta R_{b'}) - \Av(p_1,\dotsc,p_n; R_b \Delta R_{b'})| < \varepsilon$$ for every
$b, b'\in W$. 
For $I\subseteq [n]$ let $C_I\subseteq W$ be the set 
 \[ C_I=\{
b\in W \colon p_i\in R_b \Leftrightarrow i\in I \}. \]  Clearly each $C_I\in \mathcal{B}_W$,     the
sets $C_I, I \subseteq [n],$ cover $W$ and for every $I\subseteq [n]$
and $b,b'\in C_I$ we have $\mu(R_b \Delta R_{b'}) < \varepsilon$.  Hence, for any $b,b' \in C_I$ we have
$$|h_{R,A}(b) - h_{R,A}(b')| \leq \mu(A \cap (R_b \Delta R_{b'})) \leq \mu (R_b \Delta R_{b'}) < \varepsilon.$$

By Claim \ref{lem:def-sep} the function $h_{R,A}$ is $\CB_W$-integrable.
\end{proof}

\subsection{Products of finitely approximated measures}

In this section, we let  $V,W,Z$ be sets, $R \subseteq V \times W \times Z$ a relation, and let  $\CB_V, \CB_W$ be fields on $V,W$, respectively.
We assume that $\mathcal{R}_V = \{ R_{(b,c)} : (b,c) \in W \times Z \} \subseteq \CB_V$ and $\mathcal{R}_W = \{ R_{(a,c)} : (a,c) \in V \times Z \} \subseteq \CB_W$. Let $\mu$ be a measure
on $\CB_V$  and $\nu$ a  measure on $\CB_W$.

Under the above assumptions we will denote by 
$\CB_{V \times W}[R]$ the field on $V\times W$ generated by $\CB_{V} \otimes \CB_{W}$ and the sets $R_c , c \in Z$. 
Notice that  $\CB_{V \times W}[R]$ contains all $R$-definable subsets of $V \times W$.

Given measures $\mu$ and $\nu$ on $\mathcal{B}_V$ and $\mathcal{B}_W$, respectively,  in general there are
many different measures on $\mathcal{B}_{V\times W}[R]$ extending the product measure $\mu\times \nu$. In this section, in the case when at least one of $\mu$, $\nu$ is finitely approximated, we construct a certain canonical measure  extending    
$\mu\times \nu$.

Let, for example,  $\mu$ be a measure on $\mathcal{B}_V$ that is  finitely approximated on the relation $R\subseteq V\times(W\times Z)$.
By   Claim \ref{claim:fim-int}, if $E$ is an arbitrary $R$-definable subset of $V \times W$ and $A\in \mathcal{B}_V, B \in \mathcal{B}_W$, then the function $h_{E,A}\colon B \to \RR$, given by $h_{E,A}(b)=\mu(A\cap E_b)=\int_A \mathbf{1}_E(x,b) d \mu$, is $\CB_W$-integrable. Hence the double integral 
$$ \int_B \left( \int_A  \mathbf{1}_E(x,y) d \mu \right) d \nu$$
is well defined for any $A \in \mathcal{B}_V, B\in \mathcal{B}_W$.

Similarly, if $\nu$ is  a measure on $\mathcal{B}_U$ that is  finitely approximated on the relation $R$ viewed as a relation on $W\times(V\times Z)$. then for any $A \in \mathcal{B}_V, B\in \mathcal{B}_W$ the double integral 
$$ \int_A \left( \int_B  \mathbf{1}_E(x,y) d \nu \right) d \mu$$
is well defined as well.


\begin{prop}\label{prop:prod-fim}
\begin{enumerate}
\item  Let $\mu$ be a measure on $\mathcal{B}_V$ that is  finitely approximated on the relation $R\subseteq V\times(W\times Z)$.
  There is a unique  measure $\omega$ on  $\CB_{V\ttimes W}[R]$
  whose restriction to $\CB_{V} \otimes \CB_{W}$  is $\mu\ttimes \nu$ and such that $\omega (E \cap (A \times B)) =\int_B  \int_A  \mathbf{1}_E(x,y) d \mu  d \nu$ for every $R$-definable $E \subseteq V \times W$, $A \in \mathcal{B}_V, B \in \mathcal{B}_W$. 
  We denote this measure by $\mu\lltimes \nu$.
\item  If $\nu$ is a measure on $\mathcal{B}_W$ that is  finitely approximated on the relation $R\subseteq W\times(V\times Z)$, 
  then there  is a unique  measure, denoted by  $\mu\rtimes \nu$, on  $\CB_{V\ttimes W}[R]$
such that $\mu\rtimes \nu (E \cap (A \times B)) =\int_A  \int_B  \mathbf{1}_E(x,y) d \nu  d \mu$ for every $R$-definable $E \subseteq V \times W$, $A \in \mathcal{B}_V, B \in \mathcal{B}_W$. 
\item If $\mu$ and $\nu$ are measures on $\mathcal{B}_V$ and $\mathcal{B}_W$, respectively, both finitely approximated on the relation $R$, then $\mu\lltimes \nu$ is also finitely approximated on the relation $R \subseteq (V\times W) \times Z$ and $\mu\lltimes \nu (E) = \nu \lltimes \mu (E)$ for any $R$-definable $E\subseteq U\times V$. 
\end{enumerate}
\end{prop}
\begin{proof}

  (1) It is easy to see that every set $Y$ in $\CB_{V\ttimes W}[R]$ is a finite disjoint union of sets of the form $E_i \cap (A_i \times B_i)$ where $E_i\subseteq V \times W$ is $R$-definable  and $A_i \in \CB_{V}, B_i \in \CB_{W}$. We define
  $$\omega (Y) = \sum_i \int_{B_i}  \int_{A_i}  \mathbf{1}_{E_i}(x,y) d \mu  d \nu. $$
It is easy to check that $\omega$ is well-defined  and is a finitely additive probability measure on $\CB_{V \times W}[R]$ satisfying the requirements.
Uniqueness is straightforward from the definition of $\omega$.

(2) is identical to (1).

(3) Given $m \in \mathbb{N}$ and viewing $R$ as a binary relation on $(V \times W) \times Z$, we must show that $\mu \ltimes \nu$ is finitely approximated  on the family $\CR^m_{V\times W}$ (see Definition \ref{def: fap on R}). Notice that for any $R$-definable $E \in \CR_{V \times W}^m$ and $a \in V, b \in W$, we have $E_a \in \CR^m_W$ and $E_b \in \CR_V^m$ (for $R$ viewed as the corresponding binary relation).

Fix an arbitrary $\varepsilon>0$.
Let $p_{1},\ldots p_{n} \in V$ 
be such that $\mu \left(F \right)\approx^{\varepsilon}\mbox{Av}\left(p_{1},\ldots,p_{n};F\right)$
for all $F \in \CR^m_V$, and let $q_{1},\ldots,q_{m}\in W$ be such that $\nu \left(F'\right)\approx^{\varepsilon}\mbox{Av}\left(q_{1},\ldots,q_{m}; F'\right)$
for all $F' \in \CR^m_{W}$.

We claim that the set $\left\{ \left(p_{i},q_{j}\right):1\leq i<n,1\leq j<m\right\} $
gives a $2 \varepsilon$-approximation for $\mu\ltimes\nu$ on $\CR^m_{V \times W}$. Let $E \in \CR_{V \times W}^m$. Using linearity of integration, we
have 

\[
\mu \ltimes \nu \left(E \right)=\int_{W}\left(\int_{V}\mathbf{1}_{E}\left(v,w\right)d\mu\right)d\nu \approx^{\varepsilon} \]
\[\int_{W}\left(\frac{1}{n}\sum_{i=1}^{n}\mathbf{1}_{E_{w}}\left(p_{i}\right)\right)d\nu =
\frac{1}{n}\sum_{i=1}^{n}\left(\int_{W}\mathbf{1}_{E_{w}}\left(p_{i}\right)d\nu\right)=\]
\[\frac{1}{n}\sum_{i=1}^{n}\left(\int_{W}\mathbf{1}_{E_{p_i}}(w)d\nu \right)\approx^{\varepsilon}\frac{1}{n}\sum_{i=1}^{n}\left(\frac{1}{m}\sum_{j=1}^{m}\mathbf{1}_{E_{p_{i}}}\left(q_{j}\right)\right)=
\]

\[
=\frac{1}{nm}\sum_{1\leq i \leq n, 1 \leq j \leq m}\mathbf{1}_{E}\left(p_i, q_{j}\right)\mbox{,}
\]
so $\mu \ltimes \nu \left(E \right)\approx^{2\varepsilon}\mbox{Av}\left(\left\{ \left(p_{i},q_{j}\right):1 \leq i \leq n, 1 \leq j \leq m\right\} ;E \right)$.

The fact that $\mu\lltimes \nu (E) = \nu \lltimes \mu (E)$ follows since, by the above, for any $\varepsilon > 0$ we have

\[ \mu \ltimes \nu \left(E\right) \approx^{2 \varepsilon}
\frac{1}{n}\sum_{i=1}^{n}\left(\frac{1}{m}\sum_{j=1}^{m}\mathbf{1}_{E_{p_{i}}}\left(q_{j}\right)\right)=
\]
\[
\frac{1}{m}\sum_{j=1}^{m}\left(\frac{1}{n}\sum_{i=1}^{n}\mathbf{1}_{E_{q_j}}\left(p_{i}\right)\right)\approx^{\varepsilon}\frac{1}{m}\sum_{j=1}^{m}\left(\int_{V}\mathbf{1}_{E_{q_{j}}}\left(v\right)d\mu \right)=
\]

\[\int_{V}\left(\frac{1}{m}\sum_{j=1}^{m}\mathbf{1}_{E_{v}}\left(q_{j}\right)\right)d\mu \approx^{\varepsilon}
\int_{V}\left(\int_{W}\mathbf{1}_{E}\left(v,w\right)d\nu\right)d\mu= \]
\[\nu \ltimes \mu \left(E \right)\mbox{,}
\]

hence $\mu\lltimes \nu (E) \approx^{4 \varepsilon} \nu \lltimes \mu (E)$ for an arbitrary $\varepsilon > 0$.
\end{proof}

\begin{rem}  If $R\subseteq U\times V$ is a relation, then everything above can be applied to $R$ by viewing it  as $R\subseteq V\times U\times Z$, where $Z$ is any one-element set.
 \end{rem}

It is not hard to see that a product of finitely approximated measures satisfies the following weak Fubini property.

\begin{lem}\label{lem:fubini} Let $R \subseteq V \times U$.
Let $\mu$ be a measure
on $\CB_V$ which is finitely approximated on the relation $R \subseteq V \times U$,  and let $\nu$ be a   measure on $\CB_V$. For any $\varepsilon >0$,
if $\mu(R_a) < \varepsilon$ for all $a\in V$ then $(\mu \lltimes
\nu) (R) < \varepsilon$. 
\end{lem}

We extend products of finitely approximated measures to an arbitrary number of sets.

\begin{defn}\label{def: compatible system}
Given $k \in \mathbb{N}$, we say that $(R; V_i, \CB_i, \mu_i : 1\leq i \leq k)$ is a \emph{compatible finitely approximated system} if for all $i$ we have:
\begin{enumerate}
\item $R \subseteq V_1 \times \ldots V_k$,
\item $\CB_i$ is a field  on $V_i$,
\item $\mu_i$ is a measure on $\CB_i$,
\item $R_a \in \CB_i$ for each $a \in V_{i^c}$,
\item $\mu_i$ is finitely approximated on the relation $R$ (viewed as a binary relation on $V_i \times V_{i^\co}$).
 \end{enumerate}

\end{defn}

\begin{defn} \label{def: iterated product measure}
  Assume that $(R; V_i, \CB_i, \mu_i : 1\leq i \leq k)$ is a compatible finitely approximated system. By induction on $2\leq n\leq k$  we define  the  field  $\mathcal{B}_{1} \times \ldots \times \mathcal{B}_{n}$
  on $V_1\times \ldots \times V_n$ and the measure $\mu_1\llltimes \mu_n$ on  $\mathcal{B}_{1} \times \ldots \times \mathcal{B}_{n}$ as follows.

  On the induction step $n+1$, we set $V=V_1\times \ldots \times V_n$, $W=V_{n+1}$, $Z=V_{n+2}\times \ldots \times V_k$.  Viewing $R$ as $R\subseteq  V \times W \times Z$,  we set
 $\mathcal{B}_{1} \times \ldots \times \mathcal{B}_{n+1}= \mathcal{B}_{U\times V}[R]$, and 
    $\mu_1\llltimes \mu_{n+1} := (\mu_1\llltimes
    \mu_{n})\lltimes \mu_{n+1}$.

  Note that in particular $R \in \mathcal{B}_{1} \times \ldots \times \mathcal{B}_{k}$ and $E \in \mathcal{B}_{1} \times \ldots \times \mathcal{B}_{k}$ for every $R_{\otimes}$-definable $E \subseteq V_1 \times \ldots \times V_k$.
\end{defn}

\subsection{Approximations by rectangular sets}
\label{sec:appox-rect-sets}

\begin{prop}\label{prop:delta-approx}  Let $V,W$ be sets, $R \subseteq V \times W$ a subset, 
$\mu$ a  measure on $V$ which is finitely approximated on the relation $R$. Then for any $\varepsilon>0$ there are $R$-definable
  subsets $X_1,\dotsc, X_m \subseteq W$ partitioning $W$ such that
  for every  $i\in [m]$ and  any  $a,a'\in X_i$ we have 
$\mu(R_a\Delta R_{a'})  < \varepsilon$. 

In addition, if the family $\CR=\{R_a \colon  a\in W\}$ has VC-dimension at most $d$ then we can choose $D\subseteq V$ of size at
most $ 320d\bigl(\tfrac{1}{\varepsilon})^2$
such that every $X_i$ is an atom in the Boolean algebra of sets
$R$-definable over $D$, and $m \leq C_d(320d)^d \left( \frac{1}{\varepsilon} \right)^{2d}$ for some constant $C_d$ depending only on $d$.
\end{prop}
\begin{proof}  Let $\CR^\Delta=\{ R_a\Delta R_{a'} \colon a,a'\in
  W\}$. Since $\mu$ is finitely approximated on $R$, there are $p_1,\dotsc p_n\in V$ with
$|\mu(F)-\Av(p_1,\dotsc,p_n;F)|< \varepsilon$ for any $F\in
\CR^\Delta$.

For each $I\subseteq [n]$ let $X_I=\{ a\in W \colon p_i\in R_a
\Leftrightarrow i\in I \}$. It is easy to see that the sets $X_I,
I\subseteq [n]$ partition $W$, every $X_I$ is
$R$-definable  and for every $I\subseteq [n]$
and $a,a'\in X_I$ we have $\mu(R_a\Delta R_a')< \varepsilon$.

\medskip

Assume in addition that $\CR$ is a VC-family with VC-dimension at most
$d$.  As above we choose 
$p_1,\dotsc p_n\in V$ with
$$|\mu(F)-\Av(p_1,\dotsc,p_n;F)|< \varepsilon/2$$ for any $F\in
\CR^\Delta$.

Let $\omega$ be a  measure on $\CB_V$ given by  $\omega(X)=\Av(p_1,\dotsc,p_n;
X)$. Since $\CR$ has VC-dimension at
  most $d$,  the  family  $\CR^\Delta$  had dimension at most $10d$ (see
  \cite[Lemma 4.5]{lovasz2010regularity}), and by Fact \ref{fact:vc-theorem} we can
  choose an $\varepsilon/2$-net $D$ for $\CR^\Delta$ and $\omega$ with  
$|D| \leq 80d\frac{2}{\varepsilon}\log\frac{2}{\varepsilon}$. Clearly
\[  80d\tfrac{2}{\varepsilon}\log\tfrac{2}{\varepsilon} \leq
80d\bigl(\tfrac{2}{\varepsilon}\bigr)^2=320d\bigl(\tfrac{1}{\varepsilon})^2. \]

For each $I\subseteq D$ let $X_I=\{ a\in W \colon  R_a\cap D = I  \}$. It is easy to see that the sets $X_I,
I\subseteq D$, partition $W$ and  every $X_I$ is an atom in the Boolean algebra of all sets
$R$-definable over $D$. Let $I \subseteq D$ and $a,a'\in X_I$.  Then
$R_a\cap  D=R_{a'}\cap D$, hence $\omega(R_a\Delta R_{a'})\leq
\varepsilon/2$, and $\mu(R_a\Delta R_{a'}) < \varepsilon$.  Finally, by Fact \ref{fact:s-sh}, the number of different atoms $X_I$'s is at most $C_d |D|^d \leq	 C_d (320d)^d \left( \frac{1}{\varepsilon} \right)^{2d}$.
\end{proof}

\begin{thm}\label{thm:main} 
Let $(R; V_i, \CB_i, \mu_i : 1\leq i \leq k)$ be a compatible finitely approximated system. Then for every $\varepsilon>0$ there is 
an $R_\ootimes$-definable $A\subseteq  V_1\tttimes V_k$  with 
$$(\mu_1\llltimes \mu_k)(R\Delta A)< \varepsilon.$$

In addition, if $R$ has VC-dimension at most $d$ (see Definition \ref{def:nip}) then we can
choose $A$ to be $R_\ootimes$-definable 
over some $\vec D$ with $\|\vec D\|\leq C_{k,d}\bigl(\frac{1}{\varepsilon}\bigr)^{2(k-1)d}$,  where
 $C_{k,d}$ is a constant that  depends on $k$ and $d$ only.   
\end{thm}
\begin{proof}
  We proceed by induction on $k$. 

\noindent\textbf{The case $k=2$.}
Let $V_1,V_2$ and $R\subseteq V_1\ttimes V_2$ be given. 
Using Proposition \ref{prop:delta-approx} we can find $R$-definable
sets $X_1,\dotsc X_m$ partitioning  $V_2$ such that for every $i\in [m]$ and any
$a,a'\in X_i$ we have $\mu_1(R_a\Delta R_{a'})< \varepsilon$.

 For each $i\in [m]$ we pick $a_i\in X_i$ and let $A=\bigcup _{i\in
   [m]}  R_{a_i}\times X_i$. 
Obviously $A$ is $R_\ootimes$-definable. It is not hard to see that for every $a\in V_2$ we have $\mu_1(R_a
\Delta A_a)< \varepsilon$, hence, by Lemma \ref{lem:fubini},
$(\mu_1\lltimes \nu_2)(R\Delta A) < \varepsilon$.

Assume in addition that $R$ has VC-dimension at most $d$. Then
by Proposition \ref{prop:delta-approx}, we can assume that for some $D_2 \subseteq
V_1$ with $|D_2|\leq 320d\bigl(\tfrac{1}{\varepsilon})^2$
 every
$X_i$ is an $R$-definable atom over $D_2$ and $m \leq C_d(320d)^d(\frac{1}{\varepsilon})^{2d}$. 
Let $D_1=\{a_1,\dotsc, a_m\}$, and $\vec D=(D_1,D_2)$. Obviously $A$ is
$R_\ootimes$-definable over $\vec D$ and we can take
$C_{2,d}=C_d(320d)^d$. 

\medskip
\noindent{\bf Inductive step $k+1$.} 
Let $V_1,\dots,V_{k+1}$ and $R\subseteq V_1\tttimes V_{k+1}$ be given. 

Viewing $V_1\tttimes V_{k+1}$ as $V_{[k]}\ttimes V_{k+1}$ and using the
case of $k=2$ we  obtain $R$-definable
$X_1,\dotsc X_m$ partitioning  $V_{k+1}$ and points $a_i\in X_i$, $i\in
[m]$, such that 
for the set $A'=\bigcup _{i\in
   [m]}  R_{a_i}\times X_i$ we have  $(\mu_1\llltimes
 \mu_{k+1})(R\Delta A')< \varepsilon/2$. 

For each $i\in [m]$ let $R^i=R_{a_i}$. It is an $R$-definable subset of $V_1\tttimes
V_k$. Applying induction hypothesis to each $R^i$ we obtain
$R^i_\ootimes$-definable sets $A_i\subseteq V_1\tttimes V_k$ such that 
$(\mu_1\llltimes \mu_k)(R^i \Delta A_i)< \varepsilon/2$.  Let $A= \bigcup _{i\in
   [m]}  A_i\times X_i$. It is an $R_\ootimes$-definable set and using Proposition \ref{prop:prod-fim} and Lemma \ref{lem:fubini}, it is not hard
 to see that $$(\mu_1\llltimes \mu_{k+1})(A'\Delta A)< \varepsilon/2,$$ hence
 $(\mu_1\llltimes \mu_{k+1})(R\Delta A)< \varepsilon$, as required.

Assume in addition that $R$ has VC-dimension at most $d$. As
in the case $k=2$ we can assume that every $X_i$ is $R$-definable over $D_{k+1}\subseteq V_1 \times \dotsc \times V_k$ with 
$|D_{k+1}| \leq 320d(\frac{2}{\varepsilon})^2$ and also assume that 
\[m\leq C_d|D_{k+1}|^d \leq C_d \Bigl[
320d\bigl(\tfrac{2}{\varepsilon}\bigr)^2\Bigr]^d =C_d(1280d)^d \bigl(\tfrac{1}{\varepsilon}\bigr)^{2d}.
\]

It is easy to see that each $R^i$ has VC-dimension at
most $d$.  Applying induction hypotheses we can assume  that each  $A_i$ 
above is  $R^i_\ootimes$-definable over $\vec D^i =(D^i_1,\dotsc
D^i_k)$ with $\|\vec D^i\| \leq C_{k,d}
(\tfrac{2}{\varepsilon})^{2(k-1)d}$, where $D^i_j \subseteq \prod_{l\in
  [k]\setminus \{j\}} V_l$.  

For each $i\in [m]$ and $j\in [k]$ let $\bar D^i_j=\{  (c,a_i) \colon
c\in D^i_j \}$,  $D_j=\bigcup_{i\in [m]}  \bar D^i_j$, and $\vec
D=(D_1,\dotsc,D_{k+1})$. 

It is not hard to see that $A$ above is $R_{\otimes}$-definable over $\vec D$ and 
\[ \| \vec D\| \leq
 m  C_{k,d}
\bigl(\tfrac{2}{\varepsilon}\bigr)^{2(k-1)d}
\leq   
C_d(1280d)^d \bigl(\tfrac{1}{\varepsilon}\bigr)^{2d} C_{k,d} 2^{2(k-1)d}
\bigl(\tfrac{1}{\varepsilon}\bigr)^{2(k-1)d} =\]
\[
=C_{k+1,d} (\tfrac{1}{\varepsilon})^{2kd}.
\]
\end{proof}

\section{Definable regularity lemma for hypergraphs of bounded VC dimension}
\label{sec:appl-hypergr}

In this section we apply the product measure decomposition results from Section \ref{sec:setting} to regularity of definable hypergraphs. Our goal is to prove a stronger version of Fact \ref{lem:sz-reg-graph} for hypergraphs of bounded VC-dimension.

\subsection{Regularity Lemmas for Hypergraphs}
\label{sec:regul-lemm-hypergr}

In this paper a \emph{$k$-hypergraph} $G=(V_1,\dotsc,V_k; R)$ consists of  sets $V_1,\dotsc,V_k$ and a subset
$R\subseteq V_1\tttimes V_k$.  We do not assume that the sets $V_i,
i\in[k]$ are pairwise distinct. 

\emph{A $k$-uniform hypergraph} $G=(V;R)$ is a set $V$ with a symmetric subset 
 $R\subseteq V^k$.  Of course, every $k$-uniform hypergraph $G=(V;R)$ can be also viewed as a $k$-hypergraph $(V,\dotsc,V; R)$ that we will denote by $\tilde G$. 

For a $k$-hypergraph $G=(V_1,\dotsc,V_k; R)$  and $A_1 \subseteq
V_1,\dotsc,A_k \subseteq V_k$, we let $R(A_1,\dotsc,A_k) := R \cap A_1\tttimes A_k$.

Let $G=(V_1,\dotsc,V_k;R)$ be a $k$-hypergraph. By \emph{a rectangular
  partition of $G$} we mean a $k$-tuple $\vec
\CP=(\CP_1,\dotsc,\CP_k)$ where each $\CP_i$ is a finite partition of
$V_i$.   For a rectangular partition $\vec \CP=(\CP_1,\dotsc,\CP_k)$
we define $\| \vec \CP \|=\max\{ |\CP_i |\colon i\in [k]\}$, and for a
set $X\subseteq V_1\tttimes  V_k$ we write $X\in \vec\CP$ if
$X=X_1\tttimes X_k$ for some $X_i\in \CP_i, i\in [k]$.   We will also
write $\Sigma \subseteq \vec \CP$ to indicate that $\Sigma$ consists of subsets
$X\subseteq V_1\tttimes V_k$ with $X\in \vec \CP$.

We say that $\vec \CP$ is \emph{$R$-definable} if each
$\CP_i$ consists of $R$-definable sets.  For a tuple $\vec
D=(D_1,\dotsc, D_k)$ as in Definition \ref{defn:const} we say that
$\vec \CP$ is $R$-definable over $\vec D$ if for each $i\in [k]$
every $X\in \CP_i$ is $R$-definable over $D_i$. 

A $k$-hypergraph $G=(V_1,\dotsc,V_k;R)$ has \emph{VC-dimension at most $d$} if $R$ has VC-dimension 
 at most $d$ in the sense of Definition \ref{def:nip}.
  A $k$-uniform hypergraph  $G=(V;R)$ has 
  VC-dimension at most $d$ if the corresponding $k$-hypergraph $\tilde
  G$ is  NIP with
  VC-dimension at most $d$.

\DIFaddend \begin{defn} \label{defn: e-regular partion with dens 0-1} Let $(R; V_i, \CB_i, \mu_i : 1\leq i \leq k)$ be a compatible finitely approximated  system. 
Let $\mu :=\mu_1\llltimes \mu_k$.
Given $\varepsilon>0$, we say that an $R$-definable  rectangular partition
$\vec \CP$ of $V_1\tttimes V_k$  is \emph{$\varepsilon$-regular with $0\mh1$-densities} if 
there is $\Sigma \subseteq \vec \CP$ such that 
\[ \sum_{X\in \Sigma} \mu(X) \leq \varepsilon, \]   
and for every $X_1\tttimes X_k\in \vec \CP \setminus \Sigma$ either
\[  \mu(X_1\tttimes X_k)-\mu( R(X_1,\dotsc,X_k)) < \varepsilon
\mu(X_1 \times\dotsc \times X_k)\]
or 
\[  \mu( R(X_1,\dotsc,X_k)) < \varepsilon
\mu(X_1 \times \dotsc \times X_k).\]
\end{defn}

\begin{rem}\label{rem: 01dens for free}
Note that in Fact \ref{lem:sz-reg-graph} the condition on the density of the edges is stated not just for the sets of the form $V_i \times V_j$ with $V_i,V_j$ from the partition, but also for arbitrary sets of the form $A \times B$ with $A \subseteq V_i, B \subseteq V_j$.
However, in the case of regular partitions with $0$-$1$-densities this strengthening follows for free: if, as in Definition \ref{defn: e-regular partion with dens 0-1},  we have $|\mu(X_1\tttimes X_k)- \delta \mu( R(X_1,\dotsc,X_k)) | < \varepsilon
\mu(X_1 \times\dotsc \times X_k)$ for some $\delta \in \{0,1\}$, then for arbitrary sets $Y_i \subseteq X_i$ with $Y_i \in \CB_i$ for $i \in [k]$ we also have $|\mu(Y_1\tttimes Y_k)- \delta \mu( R(Y_1,\dotsc,Y_k)) | < \varepsilon
\mu(X_1 \times\dotsc \times X_k)$.
\end{rem}

The next theorem demonstrates how existence of an approximation by rectangular sets for the product measure proved in Section \ref{sec:setting} can be used to obtain a regular partition.

\begin{thm}
  \label{thm:main1}
Let $(R; V_i, \CB_i, \mu_i : 1\leq i \leq k)$ be a compatible finitely approximated system and  $\mu=\mu_1\llltimes \mu_k$.
Then for any $\varepsilon>0$ there is an $R$-definable
$\varepsilon$-regular partition $\vec\CP$ with $0\mh1$-densities.

In addition, if  $R$ is NIP with VC dimension at most $d$ we can
choose $\vec\CP$ with $\| \vec\CP\| \leq
C_d(C_{k,d})^d\bigl(\tfrac{1}{\varepsilon}\bigr)^{4(k-1)d^2} $, where
$C_d$ and $C_{k,d}$ are the constants from Fact \ref{fact:s-sh} and Theorem \ref{thm:main}.
\end{thm}
\begin{proof}

Using Theorem \ref{thm:main} there is a set $A$ which is $R_\ootimes$-definable over some finite set $\vec{D} = (D_1, \ldots, D_k)$
and $\mu(A\Delta R)< \varepsilon^2$. Say $A=\cup_{j\in [m]}
A^j_1\tttimes A^j_k$ where each $A^j_i \subseteq V_i$ is
$R$-definable. 

For each $i\in[k]$, let $\CP_i$ be the set of all atoms in the Boolean algebra  generated by all $R$-definable over $D_i$
subsets of $V_i$; so each $\CP_i$
consists of $R$-definable sets partitioning $V_i$. We claim that $\vec\CP$ is 
$\varepsilon$-regular  with $0\mh1$-densities.

Let 
\[\Sigma :=\{ X\in \vec \CP \colon \mu( X \cap (A\Delta R))\geq \varepsilon\mu(X) \}.\]
Since $\mu(A\Delta R)< \varepsilon^2$ and $\mu$ is finitely additive we
obtain that 
\[ \sum_{X\in \Sigma} \mu(X) \leq \varepsilon. \]   

Let $X=X_1\tttimes X_k\in \vec \CP \setminus \Sigma$. We have 
\[ \mu( X \cap (A\Delta R))<  \varepsilon\mu(X). \]
By definition of $\vec \CP$, either $X \subseteq A$ or $X\cap
A = \emptyset$. 

Assume first $X \subseteq A$, then $X \cap (A\Delta R)= X \setminus
R(X_1,\dotsc, X_k)$ and 
\[ \mu(X \setminus R(X_1,\dotsc, X_k)) =
\mu(X)-\mu(R(X_1,\dotsc,X_k)),\] 
hence  
\[ \mu(X_1\tttimes X_k)-\mu(R(X_1,\dotsc,X_k)) \leq \varepsilon
\mu(X_1\tttimes X_k).\]

If $X\cap A=\emptyset$ a similar argument shows that 
\[  \mu(R(X_1,\dotsc,X_k)) < \varepsilon
\mu(X_1 \times \dotsc \times X_k).\]

Assume in addition that $R$ is NIP with VC-dimension at most $d$. Then
using Theorem \ref{thm:main} we can assume that 
$| D_i|\leq  C_{k,d}\bigl(\frac{1}{\varepsilon}\bigr)^{4(k-1)d}$ for
$i\in[k]$, and by Fact \ref{fact:s-sh}, 
\[ |\CP_i|\leq C_d|D_i|^d\leq
C_d\Big(C_{k,d}\bigl(\tfrac{1}{\varepsilon}\bigr)^{4(k-1)d}\Big)^d=
C_d(C_{k,d})^d\bigl(\tfrac{1}{\varepsilon}\bigr)^{4(k-1)d^2}.
 \]
  \end{proof}

\begin{rem} In the case when each $V_i$ is finite the above theorem without
 the NIP part is trivial, since we can take $\CP_i$ to be the set of all
  atoms in the Boolean algebra of \textbf{all} $R$-definable
  subsets of $V_i$. 
\end{rem}

This immediately gives an analogous theorem for $k$-uniform
hypergraphs. We state it only in the NIP case.

\begin{thm}
  \label{thm:main2}

  Let $G=(V;R)$ be a $k$-uniform hypergraph with VC-dimension at most $d$, $\CB$ a Boolean algebra on $V$ containing all of the fibers of $R$, and let $\mu$ be a measure on $\CB$ which is finitely approximated on $R \subseteq V \times  V^{k-1}$.
Then for any $\varepsilon>0$ there is an $R$-definable partition $\CP$
of $V$ such that $\vec\CP=(\CP,\dotsc,\CP)$ is an 
$\varepsilon$-regular partition of the $k$-hypergraph $\tilde G=(V,\dotsc,V;R)$
 with $0\mh1$-densities relatively to the measure $\mu^k = \mu \ltimes \ldots \ltimes \mu$, and 
$| \vec \CP| \leq
C_d(kC_{k,d})^d\bigl(\tfrac{1}{\varepsilon}\bigr)^{4(k-1)d^2} $.
\end{thm}

\begin{proof}
  By assumption $(R; V_i, \CB_i, \mu_i : 1\leq i \leq k)$ is a compatible finitely approximated system, with $V_i := V$,  $\CB_i:=\CB$, $\mu_i := \mu$, and 
  $R$ viewed  as a relation $R \subseteq  V_1\tttimes V_k$. Let  $\mu^k=\mu_1\llltimes \mu_k$.

Let $\vec{\CP}$ be an $\varepsilon$-regular partition with $0\mh1$-densities as in the proof of Theorem \ref{thm:main1}, i.e.~$\CP_i$ is the set of all atoms in the Boolean algebra  generated by all $R$-definable over $D_i \subseteq V^{k-1}$
subsets of $V_i$. 

Let $D=\bigcup_{i\in [k]} D_i$. We take $\CP$ to be the set of all atoms
  in the Boolean algebra of all $R$-definable over  $D$ subsets of
  $V$. Then $(\CP,\dotsc,\CP)$ is also a $\varepsilon$-regular partition with $0\mh1$-densities as it refines $\vec{\CP}$, and by  Fact \ref{fact:s-sh}, 
\[ |\CP|\leq C_d|D|^d\leq
C_d\Big(kC_{k,d}\bigl(\tfrac{1}{\varepsilon}\bigr)^{4(k-1)d}\Big)^d=
C_d(kC_{k,d})^d\bigl(\tfrac{1}{\varepsilon}\bigr)^{4(k-1)d^2}
\]    
\end{proof}

Now we give some examples where Theorems \ref{thm:main1} and \ref{thm:main2} apply. 

\subsection{The finite case}
\label{sec:finite-case}

Let $G=(V_1,\dots,V_k;R)$ be a finite $k$-hypergraph.  
For each $i\in [k]$ let $\mu_i$ be the uniform counting measure on $V_i$,
i.e. $\mu_i(X)=\frac{|X|}{|V_i|}$, and let $\mu$ be the uniform counting measure on
  $V_1\tttimes V_k$. Then all $\mu_i$ and $\mu$ are finitely approximated  
  measures with $\mu=\mu_1\llltimes\mu_k$. Hence all the results of
  the previous section can be applied to finite $k$-hypergraphs with
  respect to the counting measures.

\begin{cor} \label{lem:sze-hyper}

Let $G=(V;R)$ be a finite $k$-uniform hypergraph. Assume that $R$ has 
VC-dimension at most $d$, as a relation on $V^k$.  

 There is a
partition $V=V_1\dcup\dotsb\dcup V_M$ for some $M\leq C_d(k C_{k,d})^d\bigl(\tfrac{1}{\varepsilon}\bigr)^{4(k-1)d^2}$, 
  numbers 
$\delta_{\vec i}\in \{0,1\}$ for $\vec i \in [M]^k$, and an exceptional set $\Sigma \subseteq
 [M]^k$ such that 
\[ \sum_{(i_1,\dotsc, i_k)\in \Sigma}  |V_{i_1}|\dotsb |V_{i_k}| \leq \varepsilon|V|^k  \]
and for each $\vec i=(i_1,\dotsc,i_k)\in [M]^k \setminus \Sigma$ we have 
\[  | \, |R(V_{i_1},\dotsc,V_{i_k})|- \delta_{\vec i}|V_{i_1}|\dotsb |V_{i_k}|\, | <\varepsilon
|V_{i_1}|\dotsb |V_{i_k}|. \]
\end{cor}
So the size of the partition $M$ depends only on $\varepsilon$, $k$ and the VC-dimension $d$, polynomially in terms of $\frac{1}{\varepsilon}$, and not on the size of the graph.

\DIFaddend \subsection{Hypergraphs definable in NIP structures} \label{sec: hypergraphs definable in NIP}
Now we discuss the model theoretic setting, which is the main motivating example for this article. For a detailed account of this setting, we refer to the introduction in \cite{distal} and to \cite{Bourbaki}.

Let $\CM$ be a first-order structure. Recall that a \emph{Keisler measure} on $M^n$ is a finitely additive probability measure on the
Boolean algebra of all definable subsets of $M^n$. Given a formula $\phi(x)$ with parameters from $M$ and
a Keisler measure $\mu$ on $M^{|x|}$, we will write $\mu(\phi(x))$ to denote $\mu(\phi(M^{|x|}))$. 
Let us fix a definable relation $E(x_1, \ldots, x_k)$, let $V_i = M^{|x_i|}$ and let $\CB_i$ be the Boolean algebra of all definable subsets of $M^{|x_i|}$. Let $\mu_i$ be a Keisler measure on $M^{|x_i|}$, equivalently a measure on $\CB_i$.

Recall that a structure $\mathcal{M}$ is \emph{an NIP structure} if for every formula
$\phi(x,y)$ the family of all $\phi$-definable sets
$\mathcal{F}_\phi=\{\phi(M,a) : a \in M^{|y|} \}$ has finite VC-dimension. In particular, if $\CM$ is NIP, then any definable relation $E(x_1, \ldots, x_k) \subseteq M^{|x_1|} \times \ldots \times M^{|x_k|}$ has finite VC dimension (in the sense of Definition \ref{def:nip}). Recall that, in an NIP structure $\CM$, a Keisler measure $\mu$ on $M^{|x|}$ is \emph{generically stable} if it is finitely approximated on \emph{all} definable relations $\phi(x,y) \subseteq M^{|x|} \times M^{|y|}$, in particular on $E$. 
\begin{rem} \label{rem: fap vs fim}
There are several equivalent characterizations of generically stable measures in NIP structures. Our definition of finitely approximated measures only requires the \emph{existence} of an $\varepsilon$-approximation for every $\varepsilon$. A stronger notion of a \emph{fim measure} is given in \cite{hrushovski2013generically} requiring that in fact for every $\varepsilon$, there is some sufficiently large $n$ such that \emph{almost all} $n$-tuples (in the sense of the product measure $\mu^{(n)}$) give an $\varepsilon$-approximation. While $\mu$ is finitely approximated on all formulas if and only if it is fim on all formulas under the NIP assumption (by the results in \cite{hrushovski2013generically}), it is not clear if the equivalence holds in general \footnote{While this paper was under review, some examples separating the two notions outside of the NIP context were provided in \cite{conant2020remarks}.}.
\end{rem}

Now, the semidirect product $\mu = \mu_1\llltimes \mu_k $ corresponds to the non-forking product $\mu_1 \otimes \ldots \otimes \mu_k$. Hence Theorem \ref{thm:main1} translates into the following.

\begin{cor} \label{cor: reg in NIP}

Let $\CM$ be NIP. For every definable relation $E\left(x_{1},\ldots,x_{k}\right)$
there is some $c=c\left(E\right)$ such that: for any $\varepsilon>0$
and any generically stable\textcolor{red}{{} }Keisler measures $\mu_{i}$
on $M^{|x_{i}|}$ there are partitions $M^{|x_i|}=\bigcup_{j<K}A_{i,j}$
and a set $\Sigma\subseteq [K] ^{k}$ such that:
\begin{enumerate}
\item $K\leq\left(\frac{1}{\varepsilon}\right)^{c}$.
\item $\mu\left(\bigcup_{\left(i_{1},\ldots,i_{k}\right)\in\Sigma}A_{1,i_{1}}\times\ldots\times A_{k,i_{k}}\right)\leq \varepsilon$,
where $\mu=\mu_{1}\otimes\ldots\otimes\mu_{k}$,
\item for all $\left(i_{1},\ldots,i_{k}\right) \notin \Sigma$ we have $$|\mu(E \cap (A_{1,i_1} \times \ldots \times A_{k,i_k})) - \delta_{\vec{i}}\mu(A_{1,i_1} \times \ldots \times A_{k,i_k})| < \varepsilon \mu(A_{1, i_1} \times \ldots \times A_{k, i_k})$$ for some $\delta_{\vec{i}} \in \{ 0,1\}$.
\item each $A_{i,j}$ is defined by an instance of an $E$-formula, with this formula depending only on $E$ and $\varepsilon$.
\end{enumerate}
\end{cor}

Theorem \ref{thm:main1} is more general however as both NIP and finite approximability are only assumed locally for $R$, and can be applied outside of the context of NIP structures.
\begin{sample}
Let $\CM$ be a pseudo-finite field, viewed as a structure in the ring language (e.g. an ultraproduct of finite fields modulo some non-principal ultrafilter). Then the ultralimit of the counting measures gives a measure on the definable sets in $\CM$. This measure is finitely approximable on all quantifier-free definable relations (by Lemma \ref{lem: epsilon-approximation for stable relations}, as it is well-known that all quantifier-free formulas in $\CM$ are stable), but not finitely approximable  for general definable relations (e.g.~because the random graph is definable). Still, Theorem \ref{thm:main1} can be applied to any quantifier-free definable relation in this situation.
\end{sample} 

We list some specific structures and Keisler measures for which Corollary \ref{cor: reg in NIP} applies to all definable relations (again, see introduction in \cite{distal} for more details).

\begin{sample} \label{ex: NIP structures}
	 Examples of NIP structures:
	\begin{enumerate}
	\item Abelian groups and modules (see e.g. \cite{tent2012course}),
\item  $\left(\mathbb{C},+,\times,0,1\right)$ (see e.g. \cite{tent2012course}),
\item Differentially closed fields (see e.g. \cite{tent2012course}),
\item free groups (in the pure group language $\left(\cdot,^{-1},0\right)$, see \cite{sela2013diophantine}),
\item Planar graphs (in the language with a single binary relation corresponding to the edges, see \cite{podewski1978stable}).
\item (Weakly) $o$-minimal structures, e.g. $M=\left(\mathbb{R},+,\times,e^{x}\right)$ (see \cite{distal}).
\item Presburger arithmetic, i.e. the ordered group of integers (see \cite{distal}).
\item $p$-minimal structures with Skolem functions,  e.g. $\left(\mathbb{Q}_{p},+,\times\right)$
for each prime $p$.
\item The (valued differential) field of transseries (\cite{aschenbrenner2015asymptotic, DistalExamples}).

\item Algebraically closed valued fields (see e.g. \cite{simon2015guide})
	\end{enumerate}
\end{sample}

\begin{sample} \label{ex: generically stable measures}
	Examples of generically stable Keisler measures (see e.g.~the introduction in \cite{distal} for more details on why these measures are generically stable):
	\begin{enumerate}
	\item Any Keisler measure concentrated on a finite set (as it is clearly finitely approximable).
  \item Let $\lambda_n$ be the Lebesgue measure on the unit cube $[0,1]^n$ in $\RR^n$.  Let
    $\mathcal M$ be an o-minimal structure expanding the field of real numbers.  If
    $X \subseteq \RR^n$ is definable in $\mathcal M$, then, by o-minimal cell decomposition,
    $X\cap [0,1]^n$ is Lebesgue measurable, hence $\lambda_n$ induces a Keisler measure on $M^n$.
  \item Similarly to (2), for every prime $p$ a (normalized) Haar measure on a compact ball in
    $\mathbb{Q}_p$ induces a Keisler measure on $\QQ_p^n$.
	\end{enumerate}
\end{sample}

\section{Stable and distal cases} \label{sec: stable and distal}

Next we consider two extreme opposite special cases of NIP hypergraphs: stable and distal ones. Stable theories are at the cornerstone of Shelah's classification theory \cite{ShelahCT}, and we refer to e.g.~\cite{tent2012course, pillay1996geometric} for a general exposition of stability. Examples (1) -- (5) in Example \ref{ex: NIP structures} are stable.
Distal theories were introduced more recently in \cite{simon2013distal} aiming to capture ``purely unstable'' structures in NIP theories. Examples (6) -- (9) in Example \ref{ex: NIP structures} are distal. Example (10) gives a combination of these two cases: it has a stable part (the algebraically closed residue field) and a distal part (the value group), and the theory developed in \cite{HHM} demonstrates that the whole structure can be analyzed in terms of these two parts. There are certain generalizations of this decomposition principle for arbitrary NIP theories \cite{shelah2012dependent, SimonTypeDecomp}.

\subsection{Stable hypergraph regularity} \label{sec: stable case}

A regularity lemma for stable graphs was proved in \cite{ms} for counting measures. Later,   \cite{malliaris2016stable} provides a proof for general measures. However, the proof in \cite{malliaris2016stable} does not give any bounds on the size of the partition. In this section we combine these two approaches and prove a regularity lemma for stable hypergraphs relatively to arbitrary measures, bounding the size of the partition by a polynomial in $\frac{1}{\varepsilon}$.

\begin{defn} \label{def: stability}
\begin{enumerate}
	\item A binary relation $R\subseteq V \times W$ is $d$-\emph{stable}, $d \in \mathbb{N}$, if there is no tree of parameters $(b_\eta: \eta \in 2^{<d})$ in $W$ such that for any $\eta \in 2^d$ there is some $a_\eta \in V$ such that $a_\eta \in R_{b_\nu} \iff \nu \frown 1 \trianglelefteq \eta $ (where $\trianglelefteq$ is the tree order). 
	\item A relation $R \subseteq V_1 \times \ldots \times V_k$ is \emph{$d$-stable} if for every $I\subseteq [k]$, viewed as a binary relation on $V_I \times V_{I^\co}$, it is $d$-stable.
	\item A relation $R$ is stable if it is $d$-stable for some $d$.
\end{enumerate}
\end{defn}

Note that if $R$ is stable, then it has finite VC-dimension.

\DIFaddend \begin{rem}
	Alternatively, stability of a relation can be defined in terms of the so called \emph{order property}. Namely, $R \subseteq V \times W$ has the $e$-order property, $e \in \mathbb{N}$, if there are some elements $a_i$ in $V$ and $b_i$ in $W$, $i =1, \ldots, e$, such that $a_i \in R_{b_j} \iff i < j$ for all $1\leq i\neq j \leq e$. It is a standard fact in basic stability theory that $R$ is stable (in the sense of Definition \ref{def: stability}) if and only if it does not have the $e$-order property for some $e$ (but the relation between the corresponding parameters $e$ and $d$ is exponential, see e.g. \cite[Lemma 6.7.9]{hodges1993model}).
\end{rem}

\begin{lem} \label{lem: epsilon-approximation for stable relations}
	Let $R \subseteq V \times W$ be a stable relation, and $\CB_V$ be a Boolean algebra on $V$ such that $R_b \in \CB_V$ for all $b \in W$. Then any measure $\mu$ on $\mathcal{B}_{V}$ is finitely approximable on $R$.
\end{lem}
\begin{proof}
Assume that  $R$ is $d$-stable.
By Definition \ref{def: fap on R}, we must show that $\mu$ is finitely approximable on the family $\mathcal{R}^m_V$ of subsets of $V$, for every $m \in \mathbb{N}$. Fix $m$.

	\textbf{Claim 1.} For any $\varepsilon > 0$ there is some $t = t(\varepsilon, d, m)$ and some $0\mh1$ measures $\delta_1, \ldots, \delta_t$ on $\CB_V$ (possibly with repetitions) such that $\mu (S) \approx^{\varepsilon} \frac{1}{t} \sum_{i=1}^{t} \delta_i(S)$ for all $S \in \mathcal{R}_V^m$.

	 \emph{Proof.}  As $R$ is stable, it follows that the family $\mathcal{R}_V^m$ has finite VC-dimension, and it depends only on $d, m$. Fix $\varepsilon > 0$. By the VC-theorem there is some $t = t(\varepsilon, d, m)$ such that for every finite (or countable) $\mathcal{F} \subseteq \mathcal{R}_V^m$ there are some $a_1, \ldots, a_t \in V$ such that $\mu (S) \approx^{\varepsilon} \frac{1}{t} \sum_{i=1}^{t} \textbf{1}_S(a_i)$ for all $S \in \mathcal{F}$. For $1\leq i \leq t$ and a finite $\CF$, define a $0\mh1$ measure $\delta^{\CF}_i$ on $\CB_V$ by $\delta^{\mathcal{F}}_i(S) := \textbf{1}_S(a_i)$ for all $S \in \CB_V$, then $\mu (S) \approx^{\varepsilon} \frac{1}{t} \sum_{i=1}^{t} \delta^{\mathcal{F}}_i(S)$ for all $S \in \mathcal{F}$. The claim now  follows by compactness of the space of all $0\mh1$ measures on $\CB_V$ (see 
\cite[Lemma 4.8]{hrushovski2011nip} 
for a more detailed account).

\textbf{Claim 2.} Every $0\mh 1$ measure $\delta$ on $\CB_V$ is finitely approximable on $\mathcal{R}_V^m$.

	\emph{Proof.}  This is a straightforward consequence of the explicit form of the definability of types in local stability. Namely, consider a binary relation $E \subseteq V \times \mathcal{R}_V^m$ given by $E := \{ (a,S) : a \in S \}$. Then $E$ is $r$-stable for some $r$ (as $R$ is stable, and stability is preserved under Boolean combinations). We can identify our $0\mh1$  measure $\delta$ restricted to $E$ with a complete $E$-type. Then (see e.g.~the proof of 
\cite[Chapter 1, Lemma 2.2]{pillay1996geometric}) for every $t \in \mathbb{N}$ we can choose some $c_1, \ldots, c_t \in V$ such that for every $S\in \mathcal{R}_V^m$:
	\begin{itemize}
	\item if $|\{i : c_i \in S \}| > r$, then $\delta(S) = 1$;
	\item if $|\{i : c_i \notin S \}| > r$, then $\delta(S) = 0$.

	\end{itemize}

Hence if $t$ is large enough so that $\frac{r}{t} < \varepsilon$, then $c_1, \ldots, c_t$ give an  $\varepsilon$-approximation of $\delta$ on $\mathcal{R}_V^m$.

~

Now, let $\varepsilon >0$ and $m$ be arbitrary, and let $\delta_1, \ldots, \delta_t$  be as given by Claim 1. By Claim 2, let $A_i$ be a multiset in $V$ giving an $\varepsilon$-approximation for $\delta_i$ on $\mathcal{R}_V^m$. It is straightforward to verify that $A = \bigcup_{i=1}^{t} A_i$ is a $2 \varepsilon$-approximation for $\mu$ on $\mathcal{R}_V^m$.
\end{proof}

\ 

From now on we work in the same setting as in Section \ref{sec:setting}. Throughout the section we let the sets $V_1, \ldots, V_k$ and a stable relation $R\subseteq V_1 \times  \ldots \times  V_k$ be given, let $\mathcal{B}_{i}$ be a field on $V_i$, and let $\mu_i$ be a measure on $\mathcal{B}_{i}$. Assume moreover that for every $i \in [k]$,  $R_b \in \mathcal{B}_{i}$ for all $b \in V_{i^\co}$.

	In view of Lemma \ref{lem: epsilon-approximation for stable relations}, if $R\subseteq V_1 \times \ldots \times V_k$ is a stable relation, then for every $I = \{ i_1, \ldots,  i_n \} \subseteq [k]$ we have a semi-direct product
measure $\mu_{I}=\mu_{i_1} \llltimes \mu_{i_n}$ on $\mathcal{B}_{I} = \mathcal{B}_{i_1} \times \ldots \times \mathcal{B}_{i_n} $ (see Definition \ref{def: iterated product measure}) which is finitely approximable on $R$ (Proposition \ref{prop:prod-fim}).

\begin{defn} \label{def: epsilon good} For any $I\subseteq [k]$ and $\varepsilon > 0$, a set $A \in \mathcal{B}_{I}$ is \emph{$\varepsilon$-good} if for any $b \in {V_{I^\co}}$, either $\mu_I(A \cap R_b) < \varepsilon \mu_I(A)$ or $\mu_I(A \cap R_b) > (1-\varepsilon) \mu_I(A)$.
\end{defn}

\begin{rem}\label{rem: e-good implies positive measure}
Notice that if a set is $\varepsilon$-good, $\varepsilon >0$, then it has measure greater than $0$.
\end{rem}

\begin{lem} \label{lem: measurable envelope for fim}
	Assume that $\mu_{I^\co}$ is finitely approximable on $R$. For any $\varepsilon > 0$, consider the set
		$$ A = \{ a \in V_I :  \mu_{I^\co}(R_a) < \varepsilon \}. $$
		Then there is an $R$-definable set $A' \supseteq A$ such that $\mu_{I^\co}(R_a) < 2 \varepsilon$ for all $a \in A'$.
\end{lem}
\begin{proof}
	Let $b_1, \ldots, b_n \in V_{I^\co}$ be such that $\mu_{I^\co}(R_a) \approx^{\frac{\varepsilon}{2}} \Av(b_1,\dotsc,b_n;R_a) $ for all $a \in V_I$. Let $\mathcal{J} = \{ J \subseteq [n] : \frac{|J|}{n} < \frac{3}{2} \varepsilon\}$, and let $A' = \bigcup_{J \in \mathcal{J}} \left( \bigcap_{j \in J} R_{b_j} \cap \bigcap_{j \notin J} \left(R_{b_j}\right)^{\co} \right)$. It is easy to check that $A'$
 satisfies the requirements. \end{proof}

\begin{lem} \label{lem: e-good definable pieces}
Fix some $I \subseteq [k]$ and some $J \subseteq [k] \setminus I$.
Let $\varepsilon>0$ and $B \in \mathcal{B}_{J}$ be an $\varepsilon$-good set, and let $A \in \mathcal{B}_{I}$ and $c \in V_{[k] \setminus (I \cup J)}$ be arbitrary, such that $A$ is of positive measure (note that $B$ is of positive measure by Remark \ref{rem: e-good implies positive measure}). Then (by Definition \ref{def: epsilon good}) $A$ is a disjoint union of the sets $$A^{0}_{B,c, \varepsilon} = \{ a \in A : \mu_{J}(R_{a,c} \cap B) < \varepsilon \mu_{J}(B) \}$$ and $$A^1_{B,c, \varepsilon} = \{ a \in A : \mu_{J}(R_{a,c} \cap B) > (1-\varepsilon) \mu_{J}(B) \}.$$ 
Assume that $\varepsilon < \frac{1}{4}$. Then $A^0_{B,c, \varepsilon}, A^1_{B,c, \varepsilon} \in \mathcal{B}_{I}$.
\end{lem}

\begin{proof}
 Indeed, let $\mu'_I$ be given by conditioning $\mu_I$ on $A$ (i.e. $\mu'_I(X) = \frac{\mu(X\cap A)}{\mu(A)}$ for all $X$) and let $\mu'_J$ be given by conditioning  $\mu_J$ on $B$. As $R$ is stable, by Lemma \ref{lem: epsilon-approximation for stable relations} both $\mu'_I, \mu'_{J}$ are finitely approximable on $R$. Hence, by Lemma \ref{lem: measurable envelope for fim} we can find some $R$-definable $A'_0 \supseteq A^0_{B,c, \varepsilon}, A'_1 \supset A^1_{B,c, \varepsilon}$ such that $\mu'_{I^\co}(R_{a,c}) < 2 \varepsilon $ for all $a \in A'_0$ and $\mu'_{I^\co}(R_{a,c}) > (1-2 \varepsilon)$ for all $a \in A'_1$ (in this case we are applying it to the complement $R^{\co}$, which is also stable). As $\varepsilon < \frac{1}{4}$, it follows that in fact $A^0_{B,c, \varepsilon} = A'_0 \cap A, A^1_{B,c, \varepsilon} = A'_1 \cap A$.
\end{proof}

In particular, it makes sense to speak of the $\mu_I$-measure of $A^0_{B,c, \varepsilon}, A^1_{B,c, \varepsilon}$.

\begin{defn} \label{def: epsilon excellent} Let $0 < \varepsilon \leq \delta < \frac{1}{4}$ be arbitrary, and let $I \subseteq [k]$. We say that a set $A \in \mathcal{B}_{I}$
 is \emph{$(\varepsilon, \delta)$-excellent} if it is $\varepsilon$-good and for every $J \subseteq [k] \setminus I$, every $\delta$-good $B \in \mathcal{B}_{J}$
 and every $c \in V_{[k] \setminus (I \cup J)}$, either $\mu_{I}(A^0_{B,c, \delta}) < \varepsilon \mu_{I}(A)$ or $\mu_{I}(A^1_{B,c, \delta}) < \varepsilon \mu_{I}(A)$ (in the notation from Lemma \ref{lem: e-good definable pieces}).
\end{defn}

\begin{rem}\label{rem: exc incr eps}
Note that if $A$ is $(\varepsilon, \delta)$-excellent, then it is also $(\varepsilon', \delta')$-excellent for any $\varepsilon' > \varepsilon$ and $\varepsilon' \leq \delta' < \delta$. However, this need not be true if we take $\delta' > \delta$ since $\delta'$-good sets need not be $\delta$-good.
\end{rem}


The following lemma is a generalization of \cite[Claim 5.4]{ms}, with an additional observation that the proof can be performed ``definably'' and with respect to an arbitrary measure.

\begin{lem} \label{lem: finding excellent subset}

  Let $R \subseteq V_1 \times \ldots \times V_k$ be $d$-stable and let $0 < \varepsilon \leq \delta < \frac{1}{2^d}$ be arbitrary.  Let $n\in [k]$.
  Assume that $A \in \mathcal{B}_{n}$ and $\mu_n(A) > 0$. Then there is an $(\varepsilon, \delta)$-excellent $R$-definable set $A' \in \mathcal{B}_{n}$ with $\mu_n(A' \cap A) \geq  \varepsilon^d \mu_n(A)$. 

\end{lem}

\begin{proof}
We will need the following claim.

\textbf{Claim.}
	Assume that $0 < \varepsilon  \leq \delta <  \frac{1}{4}$ and $A \in \mathcal{B}_{n}$ is not $(\varepsilon, \delta)$-excellent. Then there are disjoint $A^0, A^1 \subseteq A$ with $A_i \in \mathcal{B}_{n}$ and $\mu(A_i) \geq \varepsilon \mu(A)$ for $i \in \{0,1\}$, and such that for any \emph{finite} $S^0 \subseteq A^0, S^1 \subseteq A^1$ with $|S^0| + |S^1| \leq \frac{1}{\delta}$ there is some $c \in V_{n^\co}$ such that $a \in R_c$ for all $a \in S^1$ and $a \notin R_c$ for all $a \in S^0$.

\textbf{Proof.} 
If $A$ is not $\varepsilon$-good,  there is some $c \in V_{n^\co}$ such that $\mu_n(A \cap R_{c}) \geq \varepsilon \mu_{n}(A)$ and $\mu_n(A \cap (R_c)^\co) \geq \varepsilon \mu_n(A)$. We let $A^1 = A \cap R_{c}$ and $A^0 = A \cap (R_{c})^\co$.

If $A$ is $\varepsilon$-good, as it is not $(\varepsilon, \delta)$-excellent, there are some $J \subseteq [k] \setminus \{ n \}$, some set $B \in \mathcal{B}_{J}$ which is $\delta$-good, and some $c'\in V_{[k] \setminus ({n} \cup J)}$ such that $A$ is a disjoint union of the sets $A^0 := A^0_{B,c', \delta}, A^1 := A^1_{B,c', \delta}$ (in the notation from Lemma \ref{lem: e-good definable pieces}) and $\mu_n(A^t) \geq \varepsilon \mu_n(A)$ for both $t \in \{0,1\}$.
Now given $S^0, S^1$ as in the claim, we have $\mu_J(B \cap R_{a,c'}) \leq \delta \mu_J(B)$ for all $a \in S^0$ and $\mu_J(B \cap (R_{a,c'})^\co) \leq \delta \mu_J(B)$ for all $a \in S^1$. Let $$B' = B \cap (\bigcup_{a \in S^0} R_{a,c'} \cup \bigcup_{a \in S^1} (R_{a,c'})^\co).$$
 As $|S^0| + |S^1| < \frac{1}{\delta}$, it follows that $\mu_J (B') <  \frac{1}{\delta} \delta \mu_J(B) = \mu_J(B)$. In particular there is some $b' \in B \setminus B'$, and taking $c = b' \frown c'$ satisfies the claim.

Assume now that the conclusion of the lemma fails. By induction we choose sets $(A_{\eta} : \eta \in 2^{\leq d})$ in $\mathcal{B}_{n}$ such that  $A_{\emptyset} = A$ and given $\eta \in 2^{<d}$, we take $A_{\eta \frown 0} := (A_\eta)^0, A_{\eta \frown 1} := (A_\eta)^1$ as given by the claim applied to $A_\eta$. For every $\eta \in 2^d$, pick some $a_\eta \in A_\eta$ (possible as $\mu_n(A_\eta) \geq \varepsilon^d \mu_n(A) > 0$). For every $\nu \in 2^{<d}$ there is some $c_\nu \in V_{n^\co}$ such that $a_\eta \in R_{c_\nu}$ if and only if $\nu \frown 1 \trianglelefteq \eta$ --  which gives contradiction to the $d$-stability of $R$. Namely we can take $c$ given by the claim for $A_\nu$ and $S^0 = \{ a_\eta : \eta \in 2^d, \nu \frown 0 \trianglelefteq \eta \}$,$S^1 = \{ a_\eta : \eta \in 2^d, \nu \frown 1 \trianglelefteq \eta \}$ (note that $|S^0| + |S^1| \leq 2^d < \frac{1}{\delta}$ by assumption).
\end{proof}

\begin{lem} \label{lem: stable partition} Let $R \subseteq V_1 \times \ldots \times V_k$ be $d$-stable, and let $0< \varepsilon \leq \delta < \frac{1}{2^d}$ be arbitrary. For any $n \in [k]$, there is a partition of $V_n$ into
  $(\varepsilon, \delta)$-excellent  sets from $\mathcal{B}_{n}$, and the size of the partition can be
  bounded by a polynomial of degree $d+1$ in $\frac{1}{\varepsilon}$. 
\end{lem}
\begin{proof}
	Repeatedly applying Lemma \ref{lem: finding excellent subset}, we let $A_{m+1}$ be an $(\frac{\varepsilon}{2}, \delta)$-excellent subset of $B_m := V_n \setminus ( \bigcup_{1\leq i \leq m} A_i)$ with $\mu_n(A_{m+1}) \geq (\frac{\varepsilon}{2})^d \mu_n(B_m)$. Then $\mu_n(B_{m+1}) \leq \mu_n(B_{m}) - (\frac{\varepsilon}{2})^d \mu_n(B_{m}) \leq (1 - (\frac{\varepsilon}{2})^d) \mu_n(B_m)$, hence $\mu_n(B_m) \leq (1 - (\frac{\varepsilon}{2})^d)^{m}$ for all $m$. Hence we have $\mu_n(B_m) \leq \frac{\varepsilon}{2} \mu_n(A_1)$ assuming $(1 - (\frac{\varepsilon}{2})^d)^{m} \leq (\frac{\varepsilon}{2})^{d+1}$. In this case, letting $A'_1 = A_1 \cup B_m$, it is easy to check that $A'_1$ is an $(\varepsilon, \delta)$-excellent set, and $A'_1, A_2, \ldots, A_m$ is a partition of $V_n$.

	Finally, for the size of the partition, we have $ \left(1-\left( \frac{\varepsilon}{2} \right)^d \right)^m \leq \left(\frac{\varepsilon}{2}\right)^{d+1} \iff m \ln \left( 1 - \left(\frac{\varepsilon}{2} \right)^d \right) \leq (d+1) \ln \frac{\varepsilon}{2}$, and taking Taylor expansion this inequality holds provided $-m \left(\frac{\varepsilon}{2} \right)^d \leq - (d+1) \frac{1}{\left(\frac{\varepsilon}{2} \right)}$. Hence we can take $m \leq c \left( \frac{1}{\varepsilon}\right)^{d+1}$, for some $c= c(d)$.	
\end{proof}

Finally we can use the partition in Lemma \ref{lem: stable partition} to obtain a regular partition for $R\subseteq V_1\tttimes V_k$. 
\begin{lem} \label{lem: product is e-good}
Let $0 < \varepsilon \leq \delta < \frac{1}{2^d}$ be arbitrary. If $A\subseteq V_n$ is $(\varepsilon, \delta)$-excellent and $B\subseteq
V_{[n-1]}$ is $\delta$-good then $B\ttimes A$ is $(\varepsilon + \delta)$-good.  
\end{lem}
\begin{proof}
  Let $c \in V_{[n]^\co}$ be arbitrary. As $B$ is $\delta$-good and $A$ is $(\varepsilon, \delta)$-excellent, by Definition \ref{def: epsilon excellent} we have $A = A^0_{B,c, \delta} \cup A^1_{B,c, \delta}$ and either $\mu_{n}(A^0_{B,c, \delta}) < \varepsilon \mu_{n}(A)$ or $\mu_{n}(A^1_{B,c, \delta}) < \varepsilon \mu_{n}(A)$. Assume we are in the first case. Then, using the definition of $\mu_{[n]}$ and Lemma \ref{lem:fubini}, we have 
  $$ \mu_{[n]}((B \times A) \cap R_c) = \int_{A} \left( \mu_{[n-1]}(R_{a,c} \cap B) \right) d\mu_{n} \geq $$
  $$ \int_{A^1_{B,c, \delta}}  \left( \mu_{[n-1]}(R_{a,c} \cap B) \right) d\mu_n  \geq \int_{A^1_{B,c, \delta}} (1-\delta) \mu_{[n-1]}(B) d\mu_n \geq $$
  $$(1- \varepsilon) \mu_n(A) (1 - \delta) \mu_{[n-1]}(B) > (1- (\varepsilon + \delta))\mu_{[n]}(A \times B).$$
  Similarly, in the second case we obtain that $\mu_{[n]}((B \times A) \cap R_c) \leq (\varepsilon + \delta) \mu_{[n]}(A \times B)$.
\end{proof}

\begin{thm} \label{thm: abstract stable regularity}
	Let $R \subseteq V_1 \times \ldots \times V_k$ be $d$-stable, and let $0< \varepsilon < \frac{1}{2^d}$ be arbitrary. Then there is an $R$-definable
$\varepsilon$-regular partition $\vec\CP$ of $V_1 \times \ldots \times V_k$ with $0\mh1$-densities (see Definition \ref{defn: e-regular partion with dens 0-1}) without any bad $k$-tuples in the partition (i.e. $\Sigma = \emptyset$) and such that the size of the partition $\| \vec\CP \|$ is
  bounded by a polynomial of degree $d+1$ in $\frac{1}{\varepsilon}$. 
\end{thm}
\begin{proof}
	For each $n \leq k$, let $\CP_n$ be a partition of $V_n$ into $(\frac{\varepsilon}{2^{k+1}}, \frac{\varepsilon}{2})$-excellent sets as given by Lemma \ref{lem: stable partition} (in particular, $\CP_n$ has size polynomial in $\frac{1}{\varepsilon}$ as $k$ is fixed), and let $\CP := \{X_1 \times \ldots \times X_k : X_n \in \CP_n \}$.
	We claim that $\CP$ is $\varepsilon$-regular with $\Sigma = \emptyset$.
	Indeed, let $X = X_1 \times \ldots \times X_k \in \CP$ be arbitrary, and let  $X' := X_1 \times \ldots \times X_{k-1}$. Applying Lemma \ref{lem: product is e-good} $k$ times, the set $X'$ is $\frac{\varepsilon}{2}$-good, and $X_k$ is $(\frac{\varepsilon}{2}, \frac{\varepsilon}{2})$-excellent (by construction and Remark \ref{rem: exc incr eps}). Then, by Definition \ref{def: epsilon excellent}, $X_k$ is a disjoint union of the sets $X_k^0 := (X_k)^0_{X', \frac{\varepsilon}{2}}, X_k^1 := (X_k)^1_{X', \frac{\varepsilon}{2}} \in \mathcal{B}_{k}$ and $\mu_k(X_k^t) < \frac{\varepsilon}{2} \mu_k (X_k)$ for one of $t \in \{0,1\}$. 
	We have 
	$$\mu_{[k]}(R \cap X) = \int_{X_k} \mu_{[k-1]}(R_c\cap X') d \mu_k(c).$$
	As $X_k$ is a disjoint union of $X^0_k, X^1_k$ and $\mu(X^t_k) \leq \frac{\varepsilon}{2}  \mu_k(X_k)$ for some $t \in \{0,1\}$, we have
	$$ \left|\mu_{[k]}(R \cap X) - \int_{X^t_k} \mu_{[k-1]}(R_c\cap X') d \mu_k(c) \right| \leq $$
	$$ \frac{\varepsilon}{2} \mu_k (X_k) \mu_{[k-1]}(X_1 \times \ldots X_{k-1}) \leq \frac{\varepsilon}{2} \mu_{[k]}(X_1 \times \ldots \times X_k)$$
	for some $t \in \{0,1\}$.

	Assume that $t = 0$. Then for all $c \in X^0_k$ we have $\mu_{[k-1]}(R_c \cap X') < \frac{\varepsilon}{2} \mu_{[k-1]}(X')$. Hence 
	$$ \int_{X^0_k} \mu_{[k-1]}(R_c\cap X') d \mu_k(c) \leq \mu(X^0_k) \frac{\varepsilon}{2} \mu_{[k-1]}(X') \leq 
 \frac{\varepsilon}{2} \mu_{[k]}(X_1 \times \ldots \times X_k),$$
	and so $\mu_{[k]}(R \cap X) \leq \varepsilon \mu_{[k]}(X)$.

	If $t = 1$, applying the same argument to the complement $R^\co$ we obtain $\mu_{[k]}(R^\co \cap X) \leq  \varepsilon \mu_{[k]}(X)$, hence $| \mu_{[k]}(R^\co \cap X) - \mu_{[k]}(X)| \leq \varepsilon \mu_{[k]}(X)$.
\end{proof}


Similarly to Corollary \ref{cor: reg in NIP}, Theorem \ref{thm: abstract stable regularity} gives the following in the definable case. Recall that a structure $\CM$ is \emph{stable} if every binary relation definable in it is stable.

\begin{cor} \label{cor: definable stable regularity}
 Let $\CM = (M, \ldots)$ be a stable structure and $k \geq 2$. For every definable $E\left(x_{1},\ldots,x_{k}\right)$ 
there is some $c=c\left(E\right)$ such that: for any $\varepsilon>0$
and any Keisler measures $\mu_{i}$ on $M^{|x_{i}|}$ there are partitions
$M^{|x_{i}|}=\bigcup_{j<K}A_{i,j}$ satisfying:
\begin{enumerate}
\item $K\leq\left(\frac{1}{\varepsilon}\right)^{c}$;
\item for all $\vec{i} = \left(i_{1},\ldots,i_{k}\right)\in [K]^{k}$
 we have $$|\mu(E \cap (A_{1,i_1} \times \ldots \times A_{k,i_k})) - \delta_{\vec{i}}\mu(A_{1,i_1} \times \ldots \times A_{k,i_k})| < \varepsilon \mu(A_{1, i_1} \times \ldots \times A_{k, i_k})$$ for some $\delta_{\vec{i}} \in \{ 0,1\}$ (where where $\mu=\mu_{1}\otimes\ldots\otimes\mu_{k}$);
\item each $A_{i,j}$ is defined by an instance of an $E$-formula, with this formula depending
only on $E$ and $\varepsilon$.
\end{enumerate}
\end{cor}

\subsection{Distal case} \label{sec: distal case}
The class of distal theories is defined and studied in
\cite{simon2013distal}, with the aim to isolate the class of purely unstable NIP theories (as opposed to the
class of stable theories, see also \cite{simon2015guide}).
For completeness of the exposition, we recall the distal regularity lemma established in \cite{distal}, pointing out a stronger form of definability for the regular partition than the one stated there. First we recall the definition of distality (and refer to the introduction in \cite{distal} for more details).
\begin{defn} \cite{distal} \label{def: distal}
	An NIP structure $\CM$
is \emph{distal} if and only if for
every definable family $\left\{ \phi\left(x,b\right):b\in M^{d}\right\} $ of subsets of $M^{|x|}$
 there is some $t \in \mathbb{N}$ and a definable family $\left\{ \psi\left(x,c\right):c\in M^{td}\right\} $ 
such that for every $a\in M$ and every finite set $B\subset M^{d}$
there is some $c\in B^{t}$ such that $a\in\psi\left(x,c\right)$
and for every $a'\in\psi\left(x,c\right)$ we have $a'\in\phi\left(x,b\right)\Leftrightarrow a\in\phi\left(x,b\right)$,
for all $b\in B$.
\end{defn}

\begin{thm} \label{thm: distal regularity}
Let $\CM$ be distal and $k \geq 2$. For every definable $E\left(x_{1},\ldots,x_{k}\right)$, defined by an instance of some formula $\theta(x_1, \ldots, x_k; z)$, 
there is some $c=c\left( \theta \right)$ such that: for any $\varepsilon>0$
and any generically stable\textcolor{red}{{} }Keisler measures $\mu_{i}$
on $M^{|x_{i}|}$ there are partitions $M^{|x_i|}=\bigcup_{j<K}A_{i,j}$
and a set $\Sigma\subseteq\left\{ 1,\ldots,K\right\} ^{k}$ such that
\begin{enumerate}
\item $K\leq\left(\frac{1}{\varepsilon}\right)^{c}$.
\item $\mu\left(\bigcup_{\left(i_{1},\ldots,i_{k}\right)\in\Sigma}A_{1,i_{1}}\times\ldots\times A_{k,i_{k}}\right) \leq \varepsilon$,
where $\mu=\mu_{1}\otimes\ldots\otimes\mu_{k}$.
\item for all $\left(i_{1},\ldots,i_{k}\right)\notin \Sigma$, either $\left(A_{1,i_{1}}\times\ldots\times A_{k,i_{k}}\right)\cap E=\emptyset$
or $A_{1,i_{1}}\times\ldots\times A_{k,i_{k}}\subseteq E$.
\item Each $A_{i,j}$ is defined by an instance of a formula $\psi_{i}\left(x_{i},z_i \right)$
which only depends on $\theta$ (and \textbf{not on $\varepsilon$}!).
\end{enumerate}
\end{thm}

\begin{proof}
This is proved in \cite[Section 5.2]{distal}, except for the fact that in (4) the formulas $\psi_{i}\left(x_{i},z_i \right)$ can be chosen independently of $\varepsilon$ --- and we explain how to modify the proof there to obtain it. Namely, the proof of \cite[Proposition 5.3]{distal} shows that, under the assumptions of the lemma, for each $i=1, \ldots, k$ we can find a finite set of formulas $\Delta_i$ and a constant $c \in \mathbb{N}$ depending \emph{only on $\theta$} (in view of \cite[Corollary 4.6]{distal}), a finite set of parameters $A_N$ depending on $\theta$ and $\varepsilon$ with $|A_N| \leq \left( \frac{1}{\varepsilon} \right)^c$, and partitions $\mathcal{P}_i = \{ A_{i,j} : j < K \}$ of $M^{|x_i|}$ satisfying the conclusion of the lemma, except for the bold font part, such that each $A_{i,j}$ is $\Delta_i$-definable over $A_N$. 

Let $\mathcal{Q}_i$ be a partition of $M^{|x_i|}$ into the sets of realizations of complete $\Delta_i$-types over $A_N$. By distality of $\CM$, let $\Delta'_i$ be a finite set of formulas such that for every $\phi \in \Delta_i$ it contains a formula $\psi$ as in Definition \ref{def: distal}. Let $\psi_i(x_i, z_i)$ be a conjunction of all formulas in $\Delta'_i$. Then for every $a \in M^{|x_i|}$ there is a single instance $\psi_i(x_i, e)$ such that its parameters $e$ are all from $A_N$ and such that $\psi_i(x_i,e)$ isolates the complete $\Delta_i$-type of $a$ over $A_N$. Using this, we can choose a partition $\mathcal{Q}'_i$ of $M^{|x_i|}$ which refines $\mathcal{Q}_i$ (and so also refines $\mathcal{P}_i$) and such that every set in $\mathcal{Q}'_i$ is defined by an instance of $\phi_i(x_i,z_i)$ over $A_N$. Then the size of $\mathcal{Q}'_i$ is bounded by $|A_N|^{|z_i|} \leq \left( \frac{1}{\varepsilon} \right)^{c'}$ where $c' = c |z_i|$ only depends on $\theta$. Hence $\mathcal{Q}'_i, i=1, \ldots, k$ give the desired partition.
\end{proof}

\section{Definable variants of the Erd\H os-Hajnal and R\"odl theorems} \label{sec: large homog set}

In this section, we are concerned with the question of finding a ``large'' ``approximately homogeneous'' definable subset of a definable hypergraph. ``Large'' here refers to positive measure, relatively to a finitely approximable measure, and ``approximately homogeneous'' means that the edge density on the set is close to $0$ or $1$ (see below for precise definitions). We consider two very different situations --- ($k$-partite) $k$-hypergraphs and $k$-uniform hypergraphs (in the sense of Section \ref{sec:regul-lemm-hypergr}).

\subsection{Partitioned hypergraphs}
First we consider the ``partite'' situation. We are working in the same setting as in Section \ref{sec:regul-lemm-hypergr}.

\begin{thm} \label{thm: density approx EH}
	Let $E \subseteq V_1 \times \ldots \times V_k$ be a $k$-hypergraph of VC-dimension at most $d$. Then for every $\alpha, \varepsilon > 0$ there is some $\delta = \delta(k,d, \alpha, \varepsilon) > 0$ such that the following holds.

	Let $\CB_i$ be a field on $V_i$, and let $\mu_i$ be a measure on $\CB_i$ which is finitely approximable on $E$, for $i=1, \ldots, k$. Let $\mu=\mu_1\llltimes \mu_k$.
	Assume that $\mu(E) \geq  \alpha$. Then there are some $E$-definable sets $A_i \subseteq V_i$ such that $\mu_i(A_i) > \delta$ for all $i=1, \ldots, k$ and $d_E(A_1, \ldots, A_k) > 1-\varepsilon$ 
	(where $d_E(A_1, \ldots, A_k) = \frac{\mu \left(E \cap (A_1 \times \ldots \times A_k) \right)}{\mu_1(A_1) \cdot \ldots \cdot \mu_k(A_k)}$ denotes the $E$-density).

\end{thm}

\begin{proof}

This follows from the regularity lemma for NIP hypergraphs (Theorem \ref{thm:main1}). Let $\alpha, \varepsilon >0$ and $d \in \mathbb{N}$ be given. Let $\varepsilon' = \frac{\min\{ \alpha, \varepsilon \} }{4} > 0$. By Theorem \ref{thm:main1} there exist $c_1, c_2 \in \mathbb{R}$ depending only on $k,d$, $E$-definable partitions $V_i = \bigcup_{j=1, \ldots, n} A_{i,j}$ for each $i = 1, \ldots, k$ with $n \leq c_1 (\frac{1}{\varepsilon '})^{c_2}$, $\delta_{\vec{j}} \in \{0,1\}$ for each $\vec{j} \in [n]^k$ and $\Sigma \subseteq [n]^k$ such that 
 $\sum_{(j_1, \ldots, j_k) \in \Sigma} \mu_1(A_{1, j_1}) \cdot \ldots \cdot  \mu_k(A_{k, j_k}) < \varepsilon'$ and
 \begin{gather*}
 	\left \lvert \mu\left( E \cap (A_{1,j_1} \times \ldots \times A_{k,j_k}) \right) - \delta_{\vec{j}} \mu \left( A_{1,j_1} \times \ldots \times A_{k,j_k}) \right) \right \rvert \\
 	< \varepsilon'  \mu \left( A_{1,j_1} \times \ldots \times A_{k,j_k} \right)
 \end{gather*}
for all $\vec{j} = (j_1, \ldots, j_k) \in [n]^k \setminus \Sigma$. In particular, if $\vec{j} \in [n]^k \setminus \Sigma $, $\delta_{\vec{j}} = 1$ and $\mu(A_{1,j_1} \times \ldots \times A_{k,j_k}) > 0$ then $d_E(A_{1,j_1}, \ldots, A_{k,j_k}) > 1-\varepsilon' > 1 - \varepsilon$. 

Let $\delta := \frac{\varepsilon'}{c_1^k (\frac{1}{\varepsilon'})^{k c_2}} > 0$, it only depends on $k,d, \alpha, \varepsilon$. To prove the theorem, it is thus sufficient to show that there exists $\vec{j} \in [n]^k \setminus \Sigma$ so that $\delta_{\vec{j}} = 1$ and $\mu(A_{1,j_1} \times \ldots \times A_{k,j_k}) > \delta$ (which automatically implies $\mu_i(A_{i,j_i}) > \delta$ for all $i \in [k]$ as each $\mu_i$ takes values in $[0,1]$).
Assume that this fails. Then we have: 
\begin{gather*}
	\mu(E) = 
	 \sum_{\vec{j} \in [n]^k} \mu(E \cap (A_{1, j_1} \times \ldots \times A_{k, j_k})) \leq \\
	 \sum_{\vec{j} \in \Sigma} \mu (A_{1,j_1} \times \ldots \times A_{k,j_k}) + \\
	 \sum_{\vec{j} \in [n]^k \setminus \Sigma, \  \mu(A_{1, j_1} \times \ldots \times A_{k,j_k}) \leq \delta} \mu(A_{1,j_1} \times \ldots \times A_{k,j_k}) + \\
	 \sum_{\vec{j} \in [n]^k \setminus \Sigma,  \ \mu(A_{1,j_1} \times \ldots \times A_{k,j_k}) > \delta} \mu(A_{1,j_1} \times \ldots \times A_{k,j_k}) \varepsilon'  \leq \\
	 \varepsilon' + n^k \delta + \varepsilon' \leq 2 \varepsilon' + \left( c_1 \left(\frac{1}{\varepsilon '} \right)^{c_2} \right)^k \delta,
\end{gather*}
which by the choice of $\delta$ is at most $3\varepsilon '$ (in order to bound the third summand by $\varepsilon'$ we use that $\mu$ is a probability measure and the sets in the sum come from a partition of $V_1 \times \ldots \times V_k$). But this contradicts the assumption that $\mu(E) \geq \alpha \geq 4 \varepsilon'$.	
\end{proof}

\begin{rem}
In the special case when $\mu$ is an ultraproduct of counting measures concentrated on finite sets, this gives a density version of the well-known lemma of Erd\H os and Hajnal, see e.g. \cite[Lemma 2.1]{fox2008induced}

\end{rem}

	In particular, the result holds when $E$ is a definable relation in an NIP structure (see Section \ref{sec: hypergraphs definable in NIP}), giving uniform definability of the sets $A_i$ in terms of $E, \alpha, \varepsilon$.

In the case when $E$ is definable in a distal structure we have the following strengthening proved in \cite[Corollary 4.6]{distal}.
\begin{fact}
	Let $\CM$ be a distal structure and
  $\theta(x_1, \ldots, x_{k}, y)$ a formula.  Given $\alpha >0$ there is $\delta >0$ such that:
  for any relation $E(x_1, \ldots, x_k)$ defined by an instance of $\theta$ and any generically stable measures $\mu_i$ on
  $M^{|x_i|}$, if $\mu(R) \geq \alpha$ (where $\mu = \mu_1 \otimes \ldots \otimes \mu_k$), then there
  are definable sets $A_i \subseteq M^{|x_i|}$ with $\mu_i(A_i) \geq \delta$ for all
  $i=1, \ldots, k$ and $\prod_{i=1}^{k} A_i \subseteq R$.  Moreover, each $A_i$  can be defined by an instance of a
  formula $\psi_i(x_i,z_i)$ that depends only on $\theta$ and $\alpha$.

	\end{fact}

\subsection{Non-partitioned case}
In the non-partite case, however, it is much harder to find a large homogeneous subset (i.e.~a clique or an anti-clique), as it is well-known in combinatorics, and we give some examples in the definable setting illustrating it.

The following is a classical result of R\"odl.

\begin{fact} \label{fac: Rodl} (\cite{rodl1986universality}, see also \cite[Theorem 1.1]{fox2008induced})
For each $\varepsilon \in (0, \frac{1}{2})$ and finite graph $H$ there is some $\delta = \delta(H, \varepsilon) > 0$ such that every $H$-free graph on $n$ vertices contains an induced subgraph on at least $\delta n$ vertices with edge density either at most $\varepsilon$  or at least $1-\varepsilon$.
\end{fact}  


We consider a generalization of this property to finitely approximable measures.

\begin{defn} \label{def: EH, approx EH, etc}

Let $\CM$ be a structure and let $\mathfrak{M}$ be a class of Keisler measures. Let $\mathcal{E}$ be a collection of definable (symmetric) (hyper-)graphs in (some powers of) $\CM$.
\begin{enumerate}
 \item We will say that $\mathcal{E}$ satisfies the \emph{R\"odl property} with respect to $\mathfrak{M}$ if for every $E \subseteq (M^n)^k$ in $\mathcal{E}$ and every $\varepsilon>0$ there is some $\delta > 0$ such that for every $\mu \in \mathfrak{M}$, a Keisler measure on $M^n$ which is finitely approximable on $E$, there is some definable $A \subseteq M^n$ such that $\mu(A) \geq \delta$ and the $\mu^{(k)}$-density of $E$ on $A$ is either $<\varepsilon$ or $> 1 - \varepsilon$.
 \item If in addition such an $A$ can be defined by an instance of some formula that depends only on $E$, and not on $\varepsilon$, then we say that $\mathcal{E}$ satisfies the \emph{uniform R\"odl property} with respect to $\mathfrak{M}$.
 \item We will say that $\mathcal{E}$ satisfies the \emph{strong R\"odl property} with respect to $\mathfrak{M}$ if in (1) we can find a definable $E$-homogeneous subset of positive $\mu$-measure.\end{enumerate}

\end{defn}

Fact \ref{fac: Rodl} implies that if $\mathcal{E}$ is a family of pseudofinite hypergraphs of bounded VC-dimension, then it satisfies the R\"odl property with respect to the class $\mathfrak{M}$ of pseudofinite counting measures, in the language of set theory. We give some examples showing that there is little hope in generalizing this to arbitrary generically stable measures.

\begin{sample} \label{ex: no EH for gen stab}
The strong R\"odl property does not hold for graphs definable in the field of reals, with respect to the Lebesgue measure. To see this, consider the relation $E \subseteq \mathbb{R}^2 \times \mathbb{R}^2$ defined by $(a,b) E (a'b') \iff |a-a'| < |b-b'|$, and let $\mu$ be the generically stable measure on $\mathbb{R}^2$ given by restricting the Lebesgue measure on $[0,1]^2$ to the definable sets. 
We claim that there is no definable $E$-homogeneous subset of $\mathbb{R}^2$ of positive measure. Indeed, any such set $A \subseteq [0,1]^2$ would have to contain an $E$-homogeneous square, and it is easy to see that this is impossible by the definition of $E$ (one can check, however, that the uniform R\"odl property is satisfied as for any $\varepsilon > 0$ we can choose a sufficiently thin vertical stripe of positive measure such that the $E$-density on it is $\varepsilon$-close to $1$).

%

\end{sample}

It may be tempting to use the NIP regularity lemma as in the partitioned case (Theorem \ref{thm: density approx EH}) to establish the R\"odl property (applying it for a symmetric relation $R \subseteq V_1 \times V_2$ with $V = V_1 = V_2$, $\mu = \mu_1 = \mu_2$). However, it doesn't work. The reason is that, given an $\varepsilon$-regular partition $A_1, \ldots, A_n$ of $V$, it is perfectly possible that all of the pairs on the diagonal $(A_i,A_i), 1 \leq i \leq n$ are bad simultaneously. Namely, if $\Sigma$ is the collection of all bad pairs, we have that $\sum_{(i,j) \in \Sigma} \mu(A_i) \mu(A_j) < \varepsilon$. On the other hand, if let's say $(A_i : 1\leq i \leq n)$ is an equipartition, we have $\sum_{1 \leq i \leq n} \mu(A_i)^2 \leq n \frac{1}{n^2} \leq \frac{1}{n}$, which can be smaller than $\varepsilon$ when $n$ is sufficiently large.
In fact, this observation suggests an idea of a counter-example to the uniform R\"odl property, which we present in the next subsection.

\subsubsection{A counterexample to the uniform R\"odl property} \label{sec: p-adics counterex}

%
%
%
%
%

Throughout this section we are working in the field of $2$-adics $\mathbb{Q}_2$, viewed as a first-order structure 
$$\CM = \left(\mathbb{Q}_2, 0,1, +, \cdot, v(x) \leq v(y), (P_n(x))_{n \in \mathbb{N}_{\geq 2}} \right)$$
 with the universe $M=\mathbb{Q}_2$ in the Macintyre language (so $v(x) \leq v(y)$ is a binary predicate comparing the $2$-adic valuations of $x$ and $y$, and $P_n(x) \iff \exists y (x = y^n)$). Let $\mu$ be the Haar measure on $\mathbb{Q}_2$ normalized on the compact ball $\mathbb{Z}_2$ (restricted to definable sets). Then $\mathbb{Q}_2$ is a distal structure, and $\mu$ is a generically stable measure (e.g.~see the introduction in \cite{distal}).
 We think of elements in $\mathbb{Q}_2$ as branches of a binary tree and define $E \subseteq M^2$ by saying that $E(x,y)$ holds if and only if $v(x-y)$ is odd (i.e.~if the branches $x$ and $y$ split at an odd level). This is a symmetric relation definable in the Macintyre's language. We estimate the $\mu^{(2)}$-density of $E$ on certain definable sets.

\begin{lem} \label{lem: density on a ball}

 Assume that $A$ is a (valuational) ball, then the density $d_E(A)$ is either $\frac{1}{3}$ or $\frac{2}{3}$ (depending on the radius of the ball).

\end{lem}

\begin{proof}

Let $A = \{ b \in \mathbb{Q}_2 : v(c - b) \geq r \}$ be a (clopen) ball with center $c$ of radius $r$, for some $r \in \mathbb{Z}$.
We think of the elements of $\mathbb{Q}_2$ as bi-infinite binary sequences (with the set of non-zero entries containing the least element). Then we can write 
  $A = \{ \tau_0 \frown \tau: \tau \in 2^{\omega} \}$ for some $\tau_0 \in 2^{\mathbb{Z}_{\leq r}}$, where ``$\frown$'' denotes the concatenation of sequences.
For each $n \in \omega$, consider the partition $A= \bigcup_{\sigma \in 2^n}A_{\sigma}$, where $A_{\sigma} = \{ \tau_0 \frown \sigma \frown \tau : \tau \in 2^{\omega} \}$. By translation-invariance of the Haar measure we have $\mu(A_\sigma) = \frac{1}{2^n} \mu(A)$ for all $\sigma \in 2^n, n \in \omega$.

We would like to calculate the density $d_E(A) = \frac{\mu^{(2)} \left(E\cap A^2 \right)}{\mu(A)^2}$. 
For $n \in \mathbb{N}_{\geq 1}$, we consider the definable set
$$F_n := \left\{(a,a') \in A^2 : \bigvee_{\sigma \in 2^{n-1}} \bigvee_{t \in \{0,1\}} \left( a \in A_{\sigma \frown t} \land  a' \in A_{\sigma \frown (t-1)} \right) \right\}.$$
From the definition we have:
\begin{itemize}
	\item if $(a,a') \in F_n$ then $v(a-a') = r+n-1$;
	\item in particular, if $r+n$ is even then $F_n \subseteq (E \cap A^2)$;
	\item if $r+n$ is odd then $F_n \subseteq  (A^2 \setminus E)$;
	\item $F_n \cap F_{n'} = \emptyset$ for any $n \neq n' \in \mathbb{N}_{\geq 1}$;
	\item $\mu^{(2)}(F_n) = \frac{1}{2^{n}} \mu(A)^2$.
	
	(Indeed, $\mu^{(2)}(F_n) = \sum_{\sigma \in 2^{n-1}} \left( \mu^{(2)}\left((A_{\sigma\frown 0} \times A_{\sigma \frown 1}) \cup (A_{\sigma\frown 1} \times A_{\sigma \frown 0})\right) \right)$, and using that $(A_{\sigma})_{\sigma \in 2^{n}}$ is a partition, $= 2^{n-1} \cdot 2 \cdot \left( \frac{1}{2^n} \mu(A) \right)^2 = \frac{1}{2^{n}} \mu(A)^2$.)
\end{itemize}

\ 

Using these observations we obtain the following estimates.

\begin{enumerate}

\item If $r$ is odd, then $$ \sum_{n \geq 1 \textrm{ odd} } \mu^{(2)}(F_n)\leq \mu^{(2)}(E \cap A) =  d_E(A) \mu(A)^2 \leq \mu(A)^2 - \left(\sum_{n \geq 1 \textrm{ even}} \mu^{(2)}(F_n)\right).$$
We have:

$$ \sum_{n \geq 1 \textrm{ odd} } \mu^{(2)}(F_n) = \sum_{n \geq 1 \textrm{ odd} } \frac{1}{2^{n}} \mu(A)^2 = \mu(A)^2 \sum_{m \geq 0} \frac{1}{2^{2m+1}} = $$ $$ = \mu(A)^2 \sum_{m \geq 0} \frac{1}{2} \frac{1}{4^m} = \mu(A)^2 \frac{1/2}{1-1/4} = \frac{2}{3} \mu(A) ^2, $$

$$ \sum_{n \geq 1 \textrm{ even} } \mu^{(2)}(F_n) = \sum_{n \geq 1 \textrm{ even} } \frac{1}{2^{n}} \mu(A)^2 = \mu(A)^2 \sum_{m \geq 1} \frac{1}{2^{2m}} = $$ $$= \mu(A)^2 \left(\sum_{m \geq 0} \frac{1}{4^m} - 1\right) = \left(\frac{1}{1-1/4} -1 \right) \mu(A)^2 = \frac{1}{3} \mu(A)^2.$$

Combining we get $d_E(A) = \frac{2}{3}$.

\item If $r$ is even, then $$ \sum_{n \geq 1 \textrm{ even} } \mu^{(2)}(F_n) \leq \mu^{(2)}(E \cap A) = d_E(A) \mu(A)^2 \leq \mu(A)^2 - \left(\sum_{n \geq 1 \textrm{ odd}} \mu^{(2)}(F_n) \right),$$ and a similar computation shows that $d_E(A) = \frac{1}{3}$.
\end{enumerate}
\end{proof}

\begin{lem} \label{lem: contains a ball}

Fix a formula $\phi(x,y)$. Then there is some $\gamma \in (0,1)$ such that: for any tuple of parameters $b \in M^{|y|}$, if $\mu(\phi(x,b)) > 0$, then $\phi(x,b)$ contains some ball $B$ with $\mu(B) \geq \gamma \mu(\phi(x,b))$.
 \end{lem}

\begin{proof}

If $\mu(\phi(x,b)) > 0$, then $\phi(x,b)$ has to be infinite. As demonstrated in the original paper of Macintyre \cite[Theorem 2]{macintyre1976definable}, every infinite definable subset of $M$ in the $2$-adics has non-empty interior. In particular it must contain some open ball of positive Haar measure.

However, to prove the claim we need a slightly more careful analysis. We recall a couple of facts about the $2$-adic cell decomposition (see e.g. \cite[Section 7]{VCD1}). Let $\phi(x,y)$ be fixed. Then there is some $N \in \mathbb{N}$, definable functions $f_i,g_i$ and elements $\lambda_i \in M$ for $i \leq N$ such that for every $b \in M^{|y|}$, the set $\phi(M,b)$ is a union of at most $N$ \emph{cells} of the form
$$U_i(b) = \{ x\in M : v(f_i(b)) \leq v(x-c_i(b)) < v(g_i(b)) \land P_{n_i}(\lambda_i(x-c_i(b))) \}.$$

Besides, we have the following fact.
\begin{fact} \label{fac: Rychlik} (see e.g. \cite[Lemma 7.4]{VCD1})
Suppose $n > 1$, and let $x,y,a \in M$ be such that $v(y-x)>2v(n)+v(y-a)$. Then $x-a$ and $y-a$ are in the same coset of $P_n$.
\end{fact}

Assume now that $\mu(\phi(x,b)) > 0$. 
Then for at least one $i \leq N$, the corresponding cell $U_i(b) \subseteq \phi(M,b)$ satisfies $\mu(U_i(b)) \geq  \frac{1}{N} \mu(\phi(x,b))$. Let $U(b) := U_i(b), f := f_i, g := g_i, n := n_i, c := c_i, \lambda := \lambda_i$.

We claim that there is some element $a \in U(b)$ with $v(a-c(b)) \leq v(f(b)) + n$. First, as the valuation group of $\mathbb{Q}_2$ is $\mathbb{Z}$, there must exist some $\beta \in \mathbb{Z}$ such that $v(f(b)) \leq \beta \leq v(f(b)) + n$ and $v(\lambda) +   \beta = n \alpha$ for some $\alpha \in \mathbb{Z}$. Let $e \in M$ be arbitrary with $v(e) = \alpha$, and let $a := \frac{e^n}{\lambda} + c(b) \in M$. Then:
\begin{enumerate}
	\item $\lambda(a-c(b)) = e^n$ (in particular $P_n(\lambda(a-c(b)))$ holds),
	\item $v(a-c(b)) = v(\frac{e^n}{\lambda}) = v(e^n) - v(\lambda) = n v(e) - v(\lambda) = n\alpha - v(\lambda) = \beta$.
\end{enumerate}  
Now either $\beta < v(g(b))$, in which case $a \in U(b)$, or $\beta \geq v(g(b))$, in which case any element in $U(b)$ satisfies the claim.

Now we consider the ball $B := B_{\geq m}(a)$ for $m := 2v(n) + v(f(b)) + n + 1$. We claim that $B \subseteq U(b)$. Indeed, for any $x \in B$ we have $v(a-x) > 2 v(n) + v(a - c(b))$, hence
by Fact \ref{fac: Rychlik}, $x-c(b)$ and $a-c(b)$ are in the same coset of $P_n$, and of course $v(x-c(b)) = v(a-c(b))$, so $x \in U(b)$.

Finally, as all balls of a given positive radius have the same Haar measure, we have $\mu(B) \geq \frac{1}{2^{2 v(n) + n + 1}} \mu(B_{\geq v(f(b))}(c(b)))$ and, as $U(b) \subseteq B_{\geq v(f(b))}(c(b))$,  we have $\mu(U(b)) \leq \mu(B_{\geq v(f(b))}(c(b)))$. Hence $$\mu(B) \geq \frac{1}{2^{2 v(n) + n + 1}} \mu(U(b)) \geq \left(\frac{1}{2^{2 v(n) + n + 1}} \cdot \frac{1}{N} \right) \mu(\phi(M,b)).$$

Note that the coefficient only depends on $\phi(x,y)$, and not on the choice of the parameter $b$.
\end{proof}

We show that the uniform R\"odl property fails for $E$. Assume towards contradiction that we can find some $\phi(x,y)$ such that for every $\varepsilon>0$ there is some set $A \subseteq M$ definable by an instance of $\phi(x,y)$ and satisfying $\mu(A) > 0$ and $d_E(A) \in [0, \varepsilon) \cup (1 - \varepsilon,1]$. Let's say $d_E(A) > 1 - \varepsilon$ (if $d_E(A) < \varepsilon$, we work with the complement of $E$ instead). Let $\gamma > 0$ be as given by Lemma \ref{lem: contains a ball} for $\phi(x,y)$, and fix some  $\varepsilon << \gamma$.

Now $A$ contains some ball $B$ with $\mu(B) = \delta \mu(A)$ for some $0 < \gamma \leq \delta \leq 1$, and we estimate the number of edges on $A$ using Lemma \ref{lem: density on a ball}.

$$\mu^{(2)}(E \cap A^2) = \mu^{(2)}(E \cap B^2) + \mu^{(2)}(E \cap (A \setminus B)^2) + 2 \mu^{(2)}(E(A\setminus B,B)) \leq $$
$$\frac{2}{3} \mu(B)^2 + \mu(A \setminus B)^2 + 2 \mu(A \setminus B) \mu(B) = $$

$$ \frac{2}{3} \delta^2 \mu(A)^2 + (1 - \delta)^2 \mu(A)^2 + 2 (1 - \delta) \delta \mu(A)^2 = $$
$$ \left(\frac{2}{3} \delta^2 + 1 - 2\delta + \delta^2 +2 \delta - 2 \delta^2 \right) \mu(A)^2 = $$

$$ \left(1 - \frac{1}{3} \delta^2    \right) \mu(A)^2  \leq \left(1 - \frac{1}{3} \gamma^2 \right) \mu(A)^2 . $$

But as we have assumed $\varepsilon << \gamma \in (0,1)$, this contradicts the assumption that $d_E(A) > 1 - \varepsilon$.

\subsubsection{Uniform R\"odl property fails for semialgebraic hypergraphs}

It is well-known that Fact \ref{fac: Rodl} fails for hypergraphs (see the example at the very end of \cite{rodl1986universality}). We observe that the uniform  R\"odl property fails already in the case of $3$-hypergraphs in the \emph{semialgebraic} setting. 

For this, let $E(x_1,x_2,x_3) \subseteq \mathbb{R}^3$ be the relation given by $(x_1<x_2<x_3) \land (x_1 + x_3 - 2x_2 \geq 0)$, it is definable in  the field of reals (it is considered in \cite[Section 3.1]{conlon2014ramsey}). We claim that it doesn't satisfy the uniform  R\"odl property relatively to the class of measures concentrated on finite sets. If we assume that it holds, then by $o$-minimality for every $\varepsilon > 0$ there is some $\delta > 0$ such that for any finite set $A \subseteq \mathbb{R}$ there is some \emph{interval} $B \subseteq \mathbb{R}$  such that for $C = A \cap B$ we have $d_E(C) > 1 - \varepsilon$ or $d_E(C) < \varepsilon$. We observe that in fact the $E$-density tends to be $\frac{1}{2}$.
Let arbitrary $\varepsilon < \frac{1}{2}$ and $\delta > 0$ be fixed.
Let us take $A = \{1, 2, 3, \ldots, N\}$ for some $N \in \mathbb{N}$ large enough (such that $\delta N$ is also large), and let $C \subseteq A, C= \{p_1, \ldots, p_n\}$ be an arbitrary interval of integers in $A$, $p_1 < \ldots < p_n$, $|C| \geq \delta N$.

Assume that $E(p_i, p_j, p_k)$ doesn't hold for some $p_1 < i < j < k < p_n$. Let us define $q_i := p_n - p_{n-i+1} + p_1$. Then we have $p_1 < q_i < q_j < q_k < p_n$, and $q_i, q_j, q_k$ are all in $C$ since $C$ is an interval. Moreover it's easy to see that $E(q_i, q_j, q_k)$ holds. This establishes a bijection between edges and non-edges in $C$, showing that the density on $C$ is arbitrary close to $1/2$ for $N$ large enough.

\bibliographystyle{acm}
\bibliography{refs}
\end{document}